\documentclass[12pt]{amsart}

\usepackage{amsmath, graphics, anysize, amssymb, float}
\marginsize{3cm}{3cm}{1cm}{3cm}

\newcommand{\nl}{\par \bigskip \noindent}
\newcommand\diag{{\sf diag}}
\newcommand\Det{{\sf det}}
\newcommand\soc{{\sf soc}}
\newcommand\Aut{{\sf Aut}} 
\newcommand\Sym{{\sf Sym}}
\newcommand\GF{{\sf GF}}
\newcommand\POmega{{\sf P\Omega}}
\newcommand\PSp{{\sf PSp}}
\newcommand\GammaL{{\sf \Gamma L}}
\newcommand\AGammaL{{\sf A\Gamma L}}
\newcommand\PGammaL{{\sf P\Gamma L}} 
\newcommand\PSigmaL{{\sf P\Sigma L}} 
\newcommand\PSL{{\sf PSL}}
\newcommand\PGL{{\sf PGL}}
\newcommand\GL{{\sf GL}}
\newcommand\SL{{\sf SL}}
\newcommand\AGL{{\sf AGL}}
\newcommand\ASL{{\sf ASL}}
\newcommand\PSU{{\sf PSU}} 
\newcommand\PSigmaU{{\sf P\Sigma U}}
\newcommand\PGammaU{{\sf P\Gamma U}}
\newcommand\PGU{{\sf PGU}}
\newcommand\PSO{{\sf PSO}}
\newcommand\Ree{{\sf Ree}}
\newcommand\Sz{{\sf Sz}} 
\newcommand\Co{{\sf Co}} 
\newcommand\HS{{\sf HS}} 
\newcommand\ord{{\sf ord}} 
\newcommand\lcm{{\sf lcm}} 
\newcommand\la{\langle}
\newcommand\ra{\rangle}
\newcommand\calB{{\mathcal B}}

\newcommand\bfZ{{\bf Z}}

\newtheorem{theorem}{Theorem}[section]%
\newtheorem{lemma}[theorem]{Lemma}%
\newtheorem{prop}[theorem]{Proposition}%
\newtheorem{corollary}[theorem]{Corollary}%
\newtheorem{hypothesis}[theorem]{Hypothesis}%
\newtheorem{remark}[theorem]{Remark}%
\newtheorem{example}[theorem]{Example}%

\begin{document}

\title[On imprimitive rank 3 permutation groups] 
{On imprimitive rank 3 permutation groups}

\author[Devillers et. al.]{Alice Devillers,
 Michael Giudici, Cai Heng Li, Geoffrey Pearce and Cheryl E.~Praeger}
\address{ Centre for the Mathematics of Symmetry and Computation\\
 School of Mathematics and Statistics\\
 The University of Western Australia\\
 35 Stirling Highway, Perth WA 6009, Australia} 
\email{alice.devillers@uwa.edu.au, michael.giudici@uwa.edu.au, cai.heng.li@uwa.edu.au, geoffrey.pearce@graduate.uwa.edu.au, cheryl.praeger@uwa.edu.au} 
\thanks{The research for this paper was supported by Australian Research Council Discovery Grant DP0770915 and Federation Fellowship Grant FF0776186. A large portion of Section 4 forms part of the PhD thesis of the fourth author, which was supported by an Australian Postgraduate Award.}
%\date{Submitted: March 12, 2010}

\maketitle
\begin{abstract}
A classification is given of rank $3$ group actions which are quasiprimitive
but not primitive. There are two infinite families and a finite number
of individual imprimitive examples. When combined with earlier work of Bannai, Kantor, Liebler,
Liebeck and Saxl, this yields a classification of all quasiprimitive rank
3 permutation groups. Our classification is achieved by
first classifying imprimitive almost simple permutation groups which
induce a $2$-transitive action on a block system and for which a block
stabiliser acts $2$-transitively on the block. We also determine those
imprimitive rank $3$ permutation groups $G$ such that the induced action on
a block is almost simple and $G$ does not contain the full socle of the
natural wreath product in which $G$ embeds.
\end{abstract}

\section{Introduction} \label{intro}

The study of rank 3 permutation groups, that is, those with exactly $3$ orbits on ordered pairs
of points goes back to the paper \cite{Higman} of Donald Higman. 
 Rank $3$  groups can be either primitive or imprimitive; however only the primitive rank $3$ groups have so far been fully classified (in \cite{Bannai},\cite{KanLieb},\cite{Liebeck}, and \cite{LiebeckSaxl}).  
The major aim of this paper is to extend this classification to include all quasiprimitive rank $3$ permutation groups (see Theorem \ref{mainthm} and Corollary \ref{cor}).
A variety of structures preserved by rank $3$ groups have been studied. Those admitting primitive rank 3 groups include symmetric designs \cite{Dempwolff01, Dempwolff04, Dempwolff06}, partial linear spaces \cite{Devillers, Devillers05}, and transitive graph decompositions \cite{gridpaper}, while those admitting imprimitive rank 3 groups as an automorphism group include  Latin square designs \cite{Devillers06} (equivalent to partial linear spaces  with 3 blocks of imprimitivity) and transitive graph decompositions \cite{Geoff} (for which Theorem \ref{ASimprimgroupthm} below is required).
As well as contributing to the classification of imprimitive rank 3
permutation groups, this paper answers a question in \cite{LDT} about complete multipartite graphs (see Remark
\ref{Knm}), and identifies and fills a gap in the proof in \cite{GodPra} of the
classification of antipodal distance transitive covers of complete graphs
(see Remark \ref{hole}).

Let $G$ be a transitive imprimitive permutation group acting on a set $\Omega$, and suppose that $G$ preserves a non-trivial block-system ${\mathcal B}$ on $\Omega$.  By the `Embedding Theorem' for imprimitive groups \cite{BMMN} we may choose a block $B \in {\mathcal B}$ and identify $\Omega$ with the set $B \times \{1,\ldots,n\}$ where $n = |{\mathcal B}|$, and then view $G$ as a subgroup of the wreath product $H \wr X$ where $G^{\mathcal B} \cong X \leq S_n$ and $H := G_B^B$, the {\em component} of $G$. 
The partition $\calB$ can be identified with $\{B\times\{i\}|i\in\{1,\ldots,n\}\}$.

 In this paper we are interested in such groups satisfying the following hypothesis.

\begin{hypothesis} \label{mainhypo}
Let $H \leq \Sym(B)$ and $X \leq S_n$, and let $G$ be a subgroup of $H \wr X$ acting imprimitively on $\Omega = B \times \{1,\ldots,n\}$, such that $G$ projects onto $X$, and $H$ is the component of $G$.  Assume that $H$ and $X$ are both $2$-transitive.
\end{hypothesis}

As we show in Lemma \ref{cc}, all imprimitive groups $G$ with a rank $3$ action on $\Omega$  satisfy Hypothesis \ref{mainhypo}.
In this paper, we  describe groups $G$ satisfying Hypothesis \ref{mainhypo} in terms of the subgroup $G \cap H^n$; that is, the kernel of the $G$-action on ${\mathcal B}$.  Analysing the groups in this way gives a characterisation in which the groups are either  very `large', or `small' and can be described in great detail.

Burnside \cite[Section 154]{Burnside} showed that a $2$-transitive permutation group $H$ has a unique minimal normal subgroup $T$ which is either a non-abelian simple group, in which case $H$ is called {\em almost simple}; or elementary abelian, in which case $H$ is called {\em affine}. Let $G$ and $H$ be as in  Hypothesis \ref{mainhypo} and let $T$ be the unique minimal normal subgroup of $H$. If $G$ has  trivial intersection with $H^n$ (so that $G$ acts faithfully on ${\mathcal B}$) we say that $G$ is {\em block-faithful with respect to ${\mathcal B}$}. If we assume that $G$ has rank 3, then Lemma \ref{uniqueblocksystem} below implies that ${\mathcal B}$ is unique, so in that case we say that $G$ is {\em block-faithful} if it is block-faithful with respect to that unique block-system. Moreover, Corollary \ref{qp} implies that $G$ is \emph{quasiprimitive} on $\Omega$, that is, all nontrivial normal subgroups of $G$ are transitive on $\Omega$.

Our first theorem gives a characterisation of rank $3$ groups when $T$ is a non-abelian simple group.
\begin{theorem} \label{ASimprimgroupthm}
Suppose that $G$ satisfies Hypothesis {\rm \ref{mainhypo}} and that $T := \soc(H)$ is a non-abelian simple group. Then $G$ has rank $3$ on $\Omega$ if and only if one of the following holds:
\begin{enumerate}
\item $T^n\leqslant G$,
\item $G$ is quasiprimitive and rank $3$ on $\Omega$,
\item $n=2$ and $G=M_{10},\PGL(2,9)$ or $\Aut(A_6)$ acting on $12$ points, or
\item $n=2$ and $G=\Aut(M_{12})$ acting on $24$ points.
\end{enumerate}
\end{theorem}

Next we classify all quasiprimitive rank 3 permutation groups that are imprimitive.

\begin{theorem} \label{mainthm}
Let $G$ be a transitive imprimitive permutation group of rank $3$ acting on a set $\Omega$. Let $n$ be the number of blocks and $m$ be the size of the blocks.
Then $G$ is quasiprimitive if and only if $G$, $n$, $m$ and $G_B^B$ are as in one of the lines of Table~$\ref{bigtable3}$.
\end{theorem}

In conjunction with \cite{Bannai,KanLieb,Liebeck,LiebeckSaxl}, Theorem \ref{mainthm}  gives a classification of all quasiprimitive rank 3 groups. 
\begin{corollary}\label{cor}
All quasiprimitive permutation groups of rank $3$ are known. They are either primitive (and classified in \cite{Bannai,KanLieb,Liebeck,LiebeckSaxl}) or imprimitive and almost simple (and classified in Theorem \ref{mainthm}).
\end{corollary}

After Theorems \ref{ASimprimgroupthm} and Theorem \ref{mainthm}, the remaining case for a complete classification of the imprimitive rank 3 permutation groups is where $H$ is a 2-transitive affine group. A discussion of this case is given in Section \ref{sec:affine}.

As a tool to prove Theorem \ref{mainthm}, we also classify all almost simple block-faithful groups $G$ satisfying Hypothesis \ref{mainhypo}. We restrict ourselves to almost simple groups as Corollary \ref{AS} implies that a block-faithful rank 3 group is almost simple. 

\begin{theorem} \label{trivkerthm}
A group $G$ which satisfies Hypothesis {\rm \ref{mainhypo}}, is almost simple and block-faithful with respect to $\mathcal{B}$, if and only if  $G$, $n$, $|B|$ are as in one of the lines of Tables~$\ref{bigtable1}$ or $\ref{bigtable2}$. 
\end{theorem}

Theorems \ref{mainthm} and \ref{trivkerthm} are proved in Section \ref{trivkerthmproof}. The tables for these theorems are given in Section \ref{sec:tables}. The proofs of our three theorems rely on the classification of almost simple $2$-transitive groups, and hence on the Classification of Finite Simple Groups.

\begin{remark}\label{Knm}{\rm
The paper \cite{LDT} studies graphs $\Gamma$ admitting a group $G$ of automorphisms that is transitive on the sets $\Gamma_i$ of vertex pairs at distance $i$, for $i\leq s$. The question (Question 4.2 of  \cite{LDT}) arose as to which complete multipartite graphs $\Gamma=K_{n[m]}$ (with $n\geq 3$) admit such a group $G$ for $s=2$, with $G$ quasiprimitive on vertices. For such $(n,m,G)$, the graph $K_{n[m]}$ admits no ``normal quotient reduction''. 
Since $s=2$ and  $K_{n[m]}$ has diameter $2$, $G$ must have rank 3, and so Theorem \ref{mainthm} gives a list of all possibilities for $(n,m,G)$, completely answering this question.}\end{remark}

\begin{remark}\label{hole}{\rm
Theorem \ref{trivkerthm}, listing the almost simple 2-transitive groups $G$ on $n$ points for which a
stabiliser has a not necessarily faithful 2-transitive action on $m$ points, is a
generalisation of Proposition 2.12 of \cite{GodPra} and reveals a missing case in that
result, namely the results table  for \cite[Proposition 2.12]{GodPra} omits the case where $G$ has socle $\PSL(3,4)$ on $n=21$ points
and a stabiliser induces a 2-transitive $\PSL(2,5)$ or $\PGL(2,5)$ of degree 6. The result in
\cite{GodPra} for which the Proposition 2.12 is applied is the classification \cite[Main
Theorem]{GodPra} of antipodal distance transitive covers of complete graphs.
This missing
case in the proposition does not lead to any additional examples in the graph
classification. This can be seen as follows. For a $6$-fold distance transitive cover of
$\mathbf{K}_{21}$ the distance transitive group must act faithfully and 2-transitively on
the antipodal partition (see the beginning of Section 5 of \cite{GodPra}) and also the
stabiliser order is divisible by the number $20\times 5=100$ of vertices at distance 2
from a given vertex. Since $|\Aut(\PSL(3,4))|$ is not divisible by 25 there is no distance
transitive group inducing a 2-transitive group with socle $\PSL(3,4)$ on the antipodal
partition.
}\end{remark}

In Section \ref{imprimrank3groups} we investigate in detail the structure of imprimitive rank $3$ permutation groups, leading to a proof in Section \ref{ASthmproof} of Theorem \ref{ASimprimgroupthm}.  Then in Section \ref{trivkerthmproof} we work towards proofs of Theorems \ref{trivkerthm} and then  \ref{mainthm}.

\section{Tables for Theorems \ref{mainthm} and \ref{trivkerthm}}\label{sec:tables}
Tables  \ref{bigtable1} and \ref{bigtable2} are organised according to the classification of almost simple $2$-transitive groups given in Theorem \ref{2transgroupsAS}. 
Moreover note the following facts:
\begin{itemize}
\item[(a)] The last two columns of Table \ref{bigtable1} are given  because that information is needed  for the proof of Theorem \ref{mainthm}. 
 \item[(b)]  On Lines $17$ and $18$ of  \ref{bigtable1}, for $n=7$ there are two representations of $A_6$ and $S_6$ on $6$ points. 
 \item[(c)] On Line $13$ of Table \ref{bigtable2}, there are two $2$-transitive representations of $G_B$ when $a\geq 4$.
\item[(d)] On Lines $6$, $9$, and $11$ of Table \ref{bigtable2}, there are up to three $2$-transitive representations of $G_B$.
\end{itemize}

\begin{center}
\begin{table}[H]
\begin{center}
\begin{tabular}{|l|l|l|l|p{7.5cm}|}
\hline
$G$ & $n$ & $m$&  $G_B^B$& extra conditions \\
\hline
$M_{11}$& $11$&$2$&$C_2$&\\
$G\geqslant \PSL(2,q)$&$q+1$&$2$&$C_2$&  $G$  satisfies the conditions explained in Proposition \ref{PSL2} \\
$G\geqslant \PSL(a,q)$& $\frac{q^a-1}{q-1}$&$m$& $\AGL(1,m)$& $a\geq 3$ and $G$ satisfies the conditions on Line $12$ of  Table $\ref{bigtable2}$, and $(md,a)=d$\\
$\PGL(3,4)$&$21$&$6$&$\PSL(2,5)$&\\
$\PGammaL(3,4)$&$21$&$6$&$\PGL(2,5)$&\\
$\PSL(3,5)$&$31$&$5$& $S_5$&\\
$\PSL(5,2)$& $31$&$8$&$A_8$&\\
$\PGammaL(3,8)$&$73$&$28$&  $\Ree(3)$&\\
$\PSL(3,2)$&$7$&$2$& $C_2$&\\
$\PSL(3,3)$&$13$&$3$&$S_3$&\\
 \hline
\end{tabular}
\caption{Quasiprimitive imprimitive rank 3 groups} \label{bigtable3}
\end{center}
\end{table}
\end{center}

\begin{center}
\begin{table}[H]
\begin{tabular}{|l|l|l|l|ll|}
\hline
Line&$G$ & $n$  &$|B|$&$G_B$ & $G_{B,B'}$\\
\hline
1&$A_7$ & $15$ &$8$, $7$& $\PSL(2,7)$ &$A_4$ \\
2&$\PSL(2,11)$ & $11$ & $6$, $5$& $A_5$ &$S_3$\\
3& $\PGammaL (2,8)$&$28$&$2$, $3$&$C_9\rtimes C_6$ &$C_2$\\
4&$\HS$&$176$& $126$&$\PSigmaU(3,5)$ &$\Aut(A_6)$\\
5&&&$2$&&\\
6&$\Co_3$&$276$ &$2$& $McL\rtimes C_2$ &$\PSU(4,3)\rtimes C_2$\\
\hline
7&$M_{11}$ & $11$ & $10$& $M_{10}$ &$C_3^2\rtimes Q_8$\\
8&&&$2$&&\\
9&$M_{11}$ & $12$ & $11$, $12$& $\PSL(2,11)$ &$A_5$\\
10&$M_{12}$ & $12$ & $11$& $M_{11}$ &$M_{10}$\\
11&&&$12$&&\\
12&$M_{22}$ & $22$ & $21$& $\PSL(3,4)$ &$C_2^4\rtimes A_5$\\
13&$M_{22}\rtimes C_2$ & $22$ & $21$& $\PSigmaL(3,4)$&$C_2^4\rtimes S_5$ \\
14&&&$2$&&\\
15&$M_{23}$ & $23$ & $22$& $M_{22}$ &$\PSL(3,4)$\\
16&$M_{24}$ & $24$ & $23$ & $M_{23}$ &$M_{22}$\\
\hline
17&$A_n$&$n$&$n-1$&$A_{n-1}$&$A_{n-2}$ \\
18&$S_n$&$n$&$n-1$&$S_{n-1}$&$S_{n-2}$\\
19&$S_n$&$n$&$2$&$S_{n-1}$&$S_{n-2}$\\
20&$S_5$&$5$&$3$&$S_{4}$&$S_{3}$\\
21&$A_6$&$6$&$6$&$A_{5}$&$A_{4}$\\
22&$S_6$&$6$&$6$&$S_{5}$&$S_{4}$\\
23&$A_7$&$7$&$10$&$A_{6}$&$A_{5}$\\
24&$S_7$&$7$&$10$&$S_{6}$&$S_{5}$\\
25&$A_8$&$8$&$15$&$A_{7}$&$A_{6}$\\
26&$A_9$&$9$&$15$&$A_{8}$&$A_{7}$\\
\hline 
27&$\PSp(2\ell,2)$&$2^{2\ell-1} - 2^{\ell-1}$&$2$&$\PSO^{-}(2\ell,2)$&$C_2^{2\ell-2}\rtimes \PSO^{-}(2\ell-2,2)$\\
28&$\PSp(2\ell,2)$&$2^{2\ell-1} + 2^{\ell-1}$&$2$&$\PSO^{+}(2\ell,2)$&$C_2^{2\ell-2}\rtimes \PSO^{+}(2\ell-2,2)$\\
29&$\PSp(6,2)$&$36$&$8$&$\PSO^{+}(6,2)\cong S_8$&$C_2^4\rtimes \PSO^{+}(4,2)$\\
\hline
\end{tabular}
\caption{Examples from groups in (a) to (d) of Theorem \ref{2transgroupsAS}} \label{bigtable1}
\end{table}
\end{center}

\begin{center}
\begin{table}[H]
\begin{tabular}{|l|l|l|l|p{8.3cm}|}
\hline
Line&$G$ & $n$ & $|B|$& extra conditions \\
\hline
1& $\unrhd\Sz(q)$&$q^2+1$&$q$& \\
\hline
2& $\unrhd\Ree(q)$&$q^3+1$&$q$& \\
3&&&$2$&\\
\hline
4& $\unrhd\PSU(3,q)$&$q^3+1$& $q^2$&  $G\cap\la \tau,\sigma\ra\leq\GammaL(1,q^2)$ transitive on $\GF(q^2)^*$ (see Rem. \ref{PSUdescr})\\%&(C)\\
5&&&$m$&$m\geq 3$ prime, $m|(q^2-1)$ and $\ord(p~\mod~m)=m-1$ \\
6&&&$2$&  $q$ is odd or $|G/(G\cap \PGU(3,q))|$ is even \\
\hline
7&$\unrhd\PSL(2,q)$&$q+1$&$q$&$G$ is 3-transitive \\
8&&&$m$&$m\geq 3$ prime, $m|(q-1)$ and $\ord(p~\mod~m)=m-1$ \\
9&&&$2$&   $q\equiv 1\pmod 4$, or  $q\equiv 3\pmod 4$ and $G\geq\PGL(2,q)$, or $|G/(G\cap \PGL(2,q))|$ is even\\
\hline
10&$\unrhd\PSL(a,q)$&$\frac{q^a-1}{q-1}$&$q^{a-1}$& \\
11&&&$2$& $|G/(G \cap\PGL(a,q))|$ is even \\
12& &&$m$&$m$ prime, $\ord(p^i\mod m)=m-1$, $dm|(q-1)$,  
 $dm| (r+\lambda d)\frac{q-1}{p^i-1}$ for some $\lambda\in [0,m-1]$,
  where $i,r,d$ are as in Remark \ref{PSLadescr}\\
13& &&$\frac{q^{a-1}-1}{q-1}$& \\
14&$\unrhd\PSL(3,4)$&$21$&$6$& \\
15&$\PSL(3,5)$&$31$&$5$& \\
16&$\PSL(5,2)$&$31$&$8$& \\
17&$\PGammaL(3,8)$&$73$&$28$&\\
18&$\PSL(4,2)$&$15$&$8$& \\
19&$\PSL(3,2)$&$7$&$2$& \\
20&$\PSL(3,3)$&$13$&$3$& \\
\hline
\end{tabular}
\caption{Examples from groups in (e) to (h) of Theorem \ref{2transgroupsAS}} \label{bigtable2}
\end{table}
\end{center}

%%%%%%%%%%%%%%%%%%%%%%%%%%%%%%%%%%%%%%%%%%%%%%%%%%%%%%%%%%%%%%%%%%%%%%%%%%%%%%%%%%%%%%%%%%%%%%%%%%%%%%%%%%%%%%%%%%%%%%%%%%%%%%%%%%%%%%%%%%%%%%%%%%%%%%%%%%%%%
\section{Imprimitive rank $3$ groups} \label{imprimrank3groups}
In this section we investigate rank $3$ subgroups $G$ of the wreath product $S_m \wr S_n$, acting imprimitively of degree $mn$. 
We identify $\Omega$ with  $B \times \{1,\ldots,n\}$ and the system of imprimitivity  with $\calB:=\{B\times\{i\}|i\in\{1,\ldots,n\}\}$, with each set $B\times\{i\}$ denoted by $B_i$.
 First, we show that such groups satisfy Hypothesis \ref{mainhypo}.

\begin{lemma} \label{cc}
Let $B$ be a set and suppose that $G \leq \Sym(B) \wr S_n$ is transitive of rank $3$ on $B \times \{1,\ldots,n\}$.  Then both $G_B^B$ and the projection $X$ of $G$ to $S_n$ are $2$-transitive.  In particular, $G$ satisfies Hypothesis $\ref{mainhypo}$ with $H := G_B^B$.
\end{lemma}
\begin{proof}
Let $\alpha \in B$.  Then $G_{(\alpha,1)}$ has exactly $2$ orbits on the remaining points of $B \times \{1,\ldots,n\}$.  Since $B_1= B \times \{1\}$ is a block for $G$, any orbit intersecting non-trivially with $B_1 \backslash \{(\alpha,1)\}$ must be fully contained in $B_1$, and it follows that one orbit is $B_1 \backslash \{(\alpha,1)\}$ while the other is $B \times \{2,\ldots,n\}$.  This implies that $(G_B^B)_\alpha$ is transitive on $B \backslash \{\alpha\}$, and $(G^{\mathcal B})_1$ is transitive on $\{2,\ldots,n\}$; and so both $G_B^B$ and $G^{\mathcal B}$ are $2$-transitive.  (Note that this holds for $n = 2$ or $|B| = 2$ provided that the group $S_2$ is considered to be $2$-transitive.)
\end{proof}

Next we give a partial converse. A group $G$ in Hypothesis \ref{mainhypo} is called \emph{full} if $G$ contains $T^n$, where $T$ is the unique minimal normal subgroup of $H$.
\begin{lemma} \label{imprimrk3lem1}
Let $B$ be a set.  Let $H \leq \Sym(B)$ and let $X \leq S_n$ for some $n$, and assume that both $H$ and $X$ are $2$-transitive.
\begin{itemize}
\item[(i)] The wreath product $H \wr X$ has rank $3$ in its imprimitive action on $B \times \{1,\ldots,n\}$.
\item[(ii)] If $G \leq H \wr X$ has as component $H$, $G$ projects onto $X$, and $G$ is full, then $G$ has rank $3$ on $B \times \{1,\ldots,n\}$.
\end{itemize}
\end{lemma}
\proof
Part (i): Let $(\alpha,1)$ be a point in $B_1$.  The stabiliser of $(\alpha,1)$ in $H \wr X$ induces $H_\alpha$ on $B_1$ and $H \wr X_1^{\{2,\ldots,n\}}$ on $B \times \{2,\ldots,n\}$.  Since $H$ is $2$-transitive, $H_\alpha$ is transitive on $B_1 \backslash \{(\alpha,1)\}$; and since $H$ is transitive and $X$ is $2$-transitive, $H \wr X_1^{\{2,\ldots,n\}}$ is transitive on $B \times \{2,\ldots,n\}$.  Thus, including the set $\{(\alpha,1)\}$, there are $3$ orbits of $( H \wr X)_{(\alpha,1)}$ on $B \times \{1,\ldots,n\}$.  Thus $H \wr X$ has rank $3$.

Part (ii): Again, the stabiliser of $(\alpha,1)$ in $G$ induces $H_\alpha$ on $B_1$, and $H_\alpha$ is transitive on $B_1 \backslash \{(\alpha,1)\}$.  Let $T = \soc(H)$, so $T^n \leq G$.  Then since $H$ is $2$-transitive and therefore primitive, $T$ is transitive on $B$.  Now $G_{(\alpha,1)}$ contains $1 \times T^{n-1}$, so $G_{(\alpha,1)} \cap T^n$ is transitive on each of the remaining blocks $B \times \{i\}$.  Furthermore $G_{(\alpha,1)}$ projects onto $X_1$ which is transitive on $\{2,\ldots,n\}$; hence $G_{(\alpha,1)}$ is transitive on $B \times \{2,\ldots,n\}$.  Thus $G_{(\alpha,1)}$ has $3$ orbits on $B \times \{1,\ldots,n\}$, and so $G$ has rank $3$.
\qed\medskip

The following Lemma is a corollary on page 147 of \cite{Higman}.
\begin{lemma} \label{uniqueblocksystem}
Suppose that $G$ is transitive but imprimitive on $\Omega$ and has rank $3$. Then $G$ admits a unique nontrivial system of imprimitivity.
\end{lemma}
It follows that, when $G$ has rank $3$ and is block-faithful with respect to $\mathcal B$, we can omit specifying $\mathcal B$ and just say that $G$  is block-faithful.

\begin{corollary} \label{AS}
Suppose that $G$ is transitive but imprimitive with block system ${\mathcal B}$, and that $G$  has rank $3$ and is block-faithful. Then $G \cong G^{\mathcal B}$ is almost simple.
\end{corollary}
\begin{proof}
The fact that $G \cong G^{\mathcal B}$ follows from the fact that $G$ acts faithfully on the set ${\mathcal B}$ of blocks.  We know by Lemma \ref{cc} that $G^{\mathcal B}$ is $2$-transitive; suppose it is of affine type.  Then $G^{\mathcal B}$ has an abelian minimal normal subgroup $N$ which acts regularly on ${\mathcal B}$.  The $N$-orbits on $\Omega$ form a partition ${\mathcal C}$ of $\Omega$ with the property that $|B \cap C| = 1$ for all $B \in {\mathcal B}$, $C \in {\mathcal C}$. Hence ${\mathcal C}$ is a nontrivial system of imprimitivity different from $\calB$, contradicting Lemma \ref{uniqueblocksystem}.  So $G^{\mathcal B}$ is not affine, and therefore (by \cite[Section 154]{Burnside}) $G^{\mathcal B}$ is almost simple.
\end{proof}

\begin{corollary} \label{qp}
Let $G$ be a transitive imprimitive permutation group of rank $3$ acting on a set $\Omega$.
Then $G$ is  block-faithful  if and only if $G$ is quasiprimitive.
\end{corollary}
\begin{proof}
Let $\calB$ be a nontrivial system of imprimitivity for $G$, and note that $\calB$ is unique by  Lemma \ref{uniqueblocksystem}. 
Since the orbits of a nontrivial normal subgroup $N$ form a system of imprimitivity of $G$, which must be $\calB$ or  $\{\Omega\}$, and since $N$ is contained in the kernel of the action of $G$ on $\calB$ in the former case, the result follows.
\end{proof}

Note that if $G$ satisfies Hypothesis \ref{mainhypo} then $G$ is not necessarily of rank $3$.  For example, the subgroup $H \times X$ of the group $H \wr X$ has rank $4$.
More precisely, the rank is at most $|B|+2$.
 This prompts the question: Given a wreath product $H \wr X$ acting imprimitively, where $H$ and $X$ are both $2$-transitive, which subgroups $G \leq H \wr X$ with $G_B^B = H$ and $G^{\mathcal B} = X$ are such that $G$ has rank $3$?  We tackle this difficult problem in the next section.

\section{Proof of Theorem \ref{ASimprimgroupthm}.} \label{ASthmproof}

We prove Theorem \ref{ASimprimgroupthm} with a series of smaller results.  Lemmas \ref{allthesame} to \ref{jhalftrans} give some general structural information about imprimitive rank $3$ groups; and then Lemmas \ref{geq3} to \ref{diagcase} deal with those groups that have an almost simple component, as in Theorem \ref{ASimprimgroupthm}.

In the following results we use $\pi_i$ to denote the projection map $H^n \longrightarrow H : (h_1,\ldots,h_n) \longmapsto h_i$.  A \emph{subdirect subgroup} of a direct product $H^n$ is a subgroup $D$ such that $\pi_i(D)=H$ for all $i$, while $D$ is also a {\em full diagonal subgroup} of $H^n$ if $D\cong H$.

\begin{lemma} \label{allthesame}
Let $G \leq H \wr X$ where $G$ projects onto $X$, and where $H$ is the component of $G$.  Assume that $X$ is transitive, and let $\hat{G}$ denote the subgroup $G \cap H^n$ of $G$.  Then $\hat{G}$ is a subdirect subgroup of $L^n$ where $L \unlhd H$; that is, $\pi_i(\hat{G})=L$ for each $i=1,\ldots,n$.
\end{lemma}
\proof  The group $\hat{G}$ is normal in $G$ since it is the kernel of the homomorphism from $G$ onto $X$.  For each $i$, write $L_i := \pi_i(\hat{G})$; so $\hat{G}$ is a subdirect subgroup of $L_1 \times \ldots \times L_n$.  Moreover, as $X$ is transitive and $H$ is the component of $G$, $L_i$ is conjugate in $H$ to $L_1$.

Now, let $h' \in H$.  Since $H$ is the component of $G$ we may assume 
without loss of generality that $H = G_{B_1}^{B_1}$ where $B_1 = B \times
\{1\}$.  But $G_{B_1}$ is a subgroup of $(H \wr X)_{B_1}$, which
factorises as $H \times (H \wr X_1^{\{2,\ldots,n\}})$.  Hence $G_{B_1}$
must project onto $H$ in the first coordinate, and it follows that there
exists an element $g' = (h',h_2,\ldots,h_n)x$ in $G_{B_1}$ (that is to
say, $1^x = 1$).
For any element $(\ell_1,\ldots,\ell_n)$ of $\hat{G}$ we have $(\ell_1,\ldots,\ell_n)^{g'} = (\ell_1^{h'},\ell_{2x^{-1}}^{h_{2x^{-1}}},\ldots,\ell_{nx^{-1}}^{h_{nx^{-1}}})$.  Since $\hat{G}$ is normal in $G$, it follows that $L_1^{h'} = L_1$.  Since $h'$ was arbitrary, $L_1 \unlhd H$.  But $L_1$ is conjugate in $H$ to $L_i$ for all $i$, and hence we have $L_i = L_1$ for all $i$.  Writing $L = L_1$, we obtain that $\hat{G}$ is a subdirect subgroup of $L^n$ where $L \unlhd H$.
\qed\medskip

\begin{lemma} \label{support}
Let $G$ be a group satisfying Hypothesis \ref{mainhypo}.  Let $T = \soc(H)$, let $\hat{G} = G \cap H^n$ and assume that $G$ is not block-faithful with respect to $\mathcal{B}$.  Then at least one of the following holds.
\begin{itemize}
\item[(i)] $G$ is full, that is $\hat{G}$ contains $T^n$;
\item[(ii)] $\hat{G}$ is a full diagonal subgroup of $L^n$ for some nontrivial normal subgroup $L$ of $H$;
\item[(iii)] $T$ is abelian.
\end{itemize}
\end{lemma}
\proof  By Lemma \ref{allthesame}, there exists a normal subgroup $L$ of $H$ such that $\hat{G}$ is a subdirect subgroup of $L^n$.  Let $g$ be a non-trivial element of $\hat{G}$ with minimal support $I$ (that is, an element for which the set $I = \{i \, | \, \pi_i(g) \neq 1\}$ has minimal cardinality).  Let $(L^n)_I$ denote the subgroup $\{a \in L^n \, | \, \pi_j(a) = 1 \; \mbox{for} \; \mbox{all} \; j \not\in I\}$ of $L^n$, and let $\hat{G}_I = \hat{G} \cap (L^n)_I$.  Then $\hat{G}_I$ is a non-trivial normal subgroup of $\hat{G}$, and since $|I|$ is minimal, $\hat{G}_I$ is a diagonal subgroup of $(L^n)_I$.
 
Observe now that if $|I| = n$, then $(L^n)_I = L^n$, and so $\hat{G}_I = \hat{G}$; that is to say, $\hat{G}$ is a full diagonal subgroup of $L^n$, and case (ii) holds.  On the other hand, if $I = \{i\}$ for some $i$, then $\hat{G}_I$ is normalised by $G_i := \{(h_1,\ldots,h_n)x \in G \, | \, i^x = i\}$.  Since the component of $G$ is $H$, it follows that $\pi_i(\hat{G}_I) \unlhd H$.  Since $H$ is $2$-transitive, $\soc(H)$ is the unique minimal normal subgroup $T$ of $H$, and hence $\pi_i(\hat{G}_I)$ contains $T$.  Since $\hat{G} \unlhd G$ and $X$ is transitive, $T^n \leq \hat{G}$ and case (i) holds.
 
So assume instead that $1 < |I| < n$, and let $i,j$ be distinct elements of $I$.  Since $X$ is $2$-transitive, there exists $t = (h_1,\ldots,h_n)x^{-1} \in G$ with $i^x = i$ and $j^x \not\in I$.  Now let $a = (a_1,\ldots,a_n)$ and $b = (b_1,\ldots,b_n)$ be elements of $\hat{G}_I$ and observe that for all $k$ we have $\pi_k([a,b^t]) = a_k^{-1} (b_{k^{x}}^{h_{k^{x}}})^{-1} a_k b_{k^{x}}^{h_{k^{x}}}$ which can be non-trivial only if both $a_k$ and $b_{k^x}$ are non-trivial.  So $\pi_k([a,b^t])$ is trivial for all $k \not\in I$; but also $\pi_j([a,b^t])$ is trivial since $b_{j^x}$ is trivial.  It follows that the support of $[\hat{G}_I,\hat{G}_I^t]$ is a proper subset of $I$.  But since $I$ is minimally non-trivial, we must have $[\hat{G}_I,\hat{G}_I^t] = 1$.  This shows in particular that $[\pi_i(\hat{G}_I),(\pi_i(\hat{G}_I))^{h_i}] = 1$.  Since $1 \neq \pi_i(\hat{G}_I) \unlhd H$ (and therefore $1 \neq \pi_i(\hat{G}_I)^{h_i} \unlhd H$ also), the groups $\pi_i(\hat{G}_I)$ and $\pi_i(\hat{G}_I)^{h_i}$ each contain $T$ and $T$ is abelian.
\qed\medskip

\nl
 Lemma \ref{support} essentially proves Theorem \ref{ASimprimgroupthm} except for specifying the exceptional examples that are listed in the theorem.  To determine these we use the next three lemmas. A permutation group is said to be {\it half-transitive} if all its orbits have the same cardinality. In particular each normal subgroup of a transitive group is half-transitive.

\begin{lemma} \label{jhalftrans}
Suppose that $G$ is a rank $3$ imprimitive subgroup of $H \wr X$ on $B \times \{1,\ldots,n\}$, where $H$ is a $2$-transitive group.  
Suppose that $G$ preserves the block system $\{B_1,\ldots,B_n\}$, where as before $B_\ell = B\times \{\ell\}$.   Suppose that $K \leq H^n$ and $K$ is a normal subgroup of $G$ and let $\alpha_i := (\alpha,i) \in B_i$ for some $i\in \{1,\ldots,n\}$.  Then, for all $j \neq i$, $\pi_j(K_{\alpha_i})$ is half-transitive on $B_j$. 
\end{lemma}
\proof Observe that for any $j \neq i$, $K_{\alpha_i}$ is a normal subgroup of $G_{\alpha_i, B_j}$.  This means that $\pi_j(K_{\alpha_i}) \unlhd \pi_j(G_{\alpha_i, B_j})$.  Since $G$ has rank 3, $\pi_j(G_{\alpha_i, B_j})$ is transitive on $B_j$, and so $\pi_j(K_{\alpha_i})$ is half-transitive.
\qed\medskip

We study groups satisfying  Hypothesis \ref{mainhypo}. We recall that we set  $B_\ell:=B\times \{\ell\}$ for $1\leq\ell\leq n$.

\begin{lemma} \label{geq3}
Suppose that $G$ is a rank $3$ group satisfying Hypothesis $\ref{mainhypo}$ such that $H$ has a non-abelian simple socle $T$.  Assume that the degree $n$ of $X$ is at least $3$, and that $G \cap T^n$ is a full diagonal subgroup of $T^n$.  Then $T$ has at least three conjugacy classes of subgroups equivalent under $\Aut T$ to the conjugacy class of point stabilisers in the $T$-action on $B$.
\end{lemma}
\proof  Let $D := G \cap T^n$, a full diagonal subgroup of $T^n$.  Then $D \unlhd G$.  For each $k$ with $2 \leq k \leq n$ there exists an automorphism $\psi_k$ of $T$ such that 
$$D = \{(t,t^{\psi_2},\ldots,t^{\psi_n}) \, | \, t \in T\}.$$
Let $i,j > 1$, and let $\alpha_i := (\alpha,i) \in B_i$ and $\alpha_j := (\alpha,j) \in B_j$.  An element $\bar{t} = (t,t^{\psi_2},\ldots,t^{\psi_n})$ of $D$ is contained in $D_{\alpha_i}$ if and only if $t^{\psi_i} \in T_\alpha$; so $D_{\alpha_i} = \{ \bar{t} \, | \, t \in T_{\alpha}^{\psi_i^{-1}}\}$.  Hence $\pi_1(D_{\alpha_i}) = T_{\alpha}^{\psi_i^{-1}}$ and $\pi_i(D_{\alpha_i}) = T_\alpha$.  Similarly $\pi_j(D_{\alpha_j}) = T_\alpha$.  Observe that $T_\alpha \neq 1$ since $H$ is $2$-transitive and almost simple, and therefore $T$ is non-regular.  Thus $|\pi_1(D_{\alpha_i})| = |T_\alpha| > 1$.  By Lemma \ref{jhalftrans}, $\pi_1(D_{\alpha_i}) = T_\alpha^{\psi_i^{-1}}$ and $\pi_1(D_{\alpha_j}) = T_\alpha^{\psi_j^{-1}}$ are both half-transitive on $B$, and nontrivial.  So neither $\pi_1(D_{\alpha_i})$ nor $\pi_1(D_{\alpha_j})$ is conjugate in $T$ to $T_\alpha$.  It remains to show that $\pi_1(D_{\alpha_i})$ and $\pi_1(D_{\alpha_j})$ are not conjugate to each other.
 
If $\pi_1(D_{\alpha_i})$ and $\pi_1(D_{\alpha_j})$ were conjugate in $T$, then $\pi_1(D_{\alpha_i})^{\psi_j} = T_\alpha^{\psi_i^{-1}\psi_j}$ would be conjugate to $\pi_1(D_{\alpha_j})^{\psi_j} = T_\alpha$.  Note that $\pi_j(D_{\alpha_i}) = T_\alpha^{\psi_i^{-1}\psi_j}$.  Thus $\pi_j(D_{\alpha_i})$ would be conjugate to $\pi_j(D_{\alpha_j}) = T_\alpha$.  Again Lemma \ref{jhalftrans} shows that this is impossible.  Hence the three subgroups $\pi_1(D_{\alpha_i})$, $\pi_1(D_{\alpha_j})$ and $\pi_i(D_{\alpha_i})$ are mutually non-conjugate, and since $\pi_i(D_{\alpha_i}) = T_\alpha$, the result follows.
\qed\medskip

\begin{lemma} \label{only2meanstrans}
Suppose that $G$ is a rank $3$ group satisfying Hypothesis \ref{mainhypo} such that $n = 2$, and assume that $\hat{G} := G \cap (H \times H)$ is a full diagonal subgroup of $H \times H$.  Then there is an automorphism $\varphi$ of $H$ such that $H_{\alpha}^\varphi$ is transitive, for some $\alpha \in B$.
\end{lemma}
\proof  There is an automorphism $\varphi$ of $H$ such that $\hat{G} = \{(h,h^\varphi) \, | \, h \in H\}$.  Since $G$ has rank $3$, the stabiliser of a point $\alpha_1 = (\alpha,1) \in B_1$ is transitive on $B_2$.  Now $G_{\alpha_1} = \hat{G}_{\alpha_1}$ and as in the proof of Lemma \ref{geq3}, $G_{\alpha_1} = \{(h,h^\varphi) \, | \, h \in H_\alpha\}$.  Thus $\pi_2(G_{\alpha_1}) = H_\alpha^\varphi$ is transitive on $B_2$.
\qed\medskip

The proof of the next lemma relies on the classification of almost simple $2$-transitive groups, and hence on the Classification of Finite Simple Groups.
\begin{lemma} \label{diagcase}
Let $G$ be a rank $3$ group satisfying Hypothesis $\ref{mainhypo}$ such that $H$ has a non-abelian simple socle $T$, and where $X$ has degree $n \geq 2$.  Assume that $G$ is neither full nor block-faithful.  Then $n = 2$ and $T$ is either $A_6$ or $M_{12}$ of degree $6$ or $12$ respectively.  Moreover, $G \cap (H \times H)$ is a full diagonal subgroup of $H \times H$.
\end{lemma}
\proof  By Lemma \ref{support}, $\hat{G} := G \cap H^n$ is a full diagonal subgroup of $L^n$ where $T \unlhd L \unlhd H$.  Table 7.4 of \cite{CameronBook} shows that $T$ has at most two conjugacy classes of subgroups equivalent under $\Aut T$ to a point stabiliser. Hence by Lemma \ref{geq3}, $n = 2$, and by Lemma \ref{only2meanstrans}, there are two such conjugacy classes.  The possible groups $T$ are $A_6$ of degree $6$, $\PSL(2,11)$ of degree $11$, $M_{12}$ of degree $12$, $A_7$ of degree $15$,  $HS$ of degree $176$, or $\PSL(m,q)$ ($m \geq 3$) of degree $(q^m - 1)/(q-1)$.  But since $n = 2$, Lemma \ref{only2meanstrans} now implies that $H$ must have a transitive subgroup equivalent under $\Aut T$ to a point stabiliser, and this eliminates all possibilities for $T$ but $A_6$ and $M_{12}$.  When $n = 2$, $\hat{G}$ is subdirect in $H \times H$ and so $L = H$.
\qed\medskip

Now we are able to prove Theorem \ref{ASimprimgroupthm}.

\nl
\emph{Proof of Theorem }\ref{ASimprimgroupthm}.  
Suppose that $G$ satisfies Hypothesis \ref{mainhypo} and has rank 3 on $\Omega$. By Corollary \ref{qp}, if $G$ is block-faithful then it is quasiprimitive. If $G$ is not quasiprimitive and we do not have $T^n\leqslant G$, then Lemmas \ref{support} and Lemma \ref{diagcase} imply that $n=2$, $T=A_6$ or $M_{12}$ of degree 6 or 12 respectively, and $\hat{G}=G\cap (H\times H)$ is a full diagonal subgroup of $H\times H$. Moreover, by Lemma \ref{uniqueblocksystem}, $\mathcal{B}$ is the unique system of imprimitivity of $G$ and so $C_G(\hat{G})=1$. Thus $G\leqslant \Aut(T)$.

Suppose first that $T=A_6$. Then $H=A_6$ or $S_6$ and $G$ is a subgroup of $\Aut(A_6)$ that has a transitive action of degree 12 with two blocks of size 6 such that the stabiliser of a point is transitive on the block not containing that point. Hence $G=M_{10}, \PGL(2,9)$ or $\Aut(A_6)$. Similarly, when $T=M_{12}$ we must have $H=M_{12}$ and $G=\Aut(M_{12})$. Conversely, we can see that all cases listed in Theorem \ref{ASimprimgroupthm} have rank 3.
\qed\medskip

\subsection{The case when $H$ is affine}\label{sec:affine}

If the $2$-transitive group $H$ in Theorem \ref{ASimprimgroupthm} is not of almost simple type, then by \cite[Section 154]{Burnside} it is of affine type.  An analogous result to Theorem \ref{ASimprimgroupthm} does not hold in this case; that is to say, when $H$ is affine with socle $T$ then $G \cap T^n$ can be a non-trival subgroup of $H^n$ which is neither $T^n$ nor a full diagonal subgroup of $T^n$.  

Suppose that $H \leq \AGL(a,q)$ for some prime power $q$ and integer $a$, and that $X$ has degree $n$.  Then $|T| = q^a$, and $T^n$ may be viewed as an $(an)$-dimensional vector space over $F := \GF(q)$.  Under the action induced by the top group $X$ of $H \wr X$ on $T^n$, $T^n$ may be viewed as an $FX$-module.  If $X \leq G$, then $\hat{T} = G \cap T^n$ corresponds to an $FX$-submodule of $T^n$.  

Mortimer \cite{Mortimer} shows that the possible $FX$-submodules (where $X$ is $2$-transitive) are too numerous to identify.  (Furthermore, we cannot even guarantee that $X$ is contained in $G$; so there may be still more possibilities for $\hat{T}$.)  However, there are four submodules which may be thought of as `standard' in the sense that they exist for every $2$-transitive group $X$.  These correspond to the following possibilities for $\hat{T}$:
\begin{itemize}
\item[(i)] $\hat{T}$ is trivial
\item[(ii)] $\hat{T}$ is a full diagonal subgroup of $T^n$
\item[(iii)] $\hat{T} = \{(t_1,\ldots,t_n) \, | \, t_1 t_2 \ldots t_n = 1\}$
\item[(iv)] $\hat{T} = T^n$.
\end{itemize}
Cases (i), (ii) and (iv) are analogous to those arising when $H$ is almost simple.  We show in the following example and lemma that case (iii) can also give rise to a rank $3$ group when $H$ is affine.

\begin{example} \label{augideal} {\em

Let $H$ be a $2$-transitive affine group with socle $T$, and let $H_0$ be the stabiliser of $0$, so that $H = T \rtimes H_0$.  Let $n$ be an integer greater than $2$, and let
$S$ be the subgroup $\{(t_1,\ldots,t_n) \, | \, t_1 t_2 \ldots t_n = 1\}$ of $T^n$.  (Note that $S$ is a subgroup because $T$ is abelian.)

Let
$$D = \diag(H_0^n) = \{(h,\ldots,h) \, | \, h \in H_0\} < H^n.$$
Let $X$ be a $2$-transitive group of degree $n$ and let $G = \langle S, D \rangle \rtimes X$.
}
\end{example}

\begin{lemma}
The group $G$ from Example $\ref{augideal}$ is a rank $3$ subgroup of $H \wr X$ with component $H$, and $G \cap T^n = S$.
\end{lemma}
\proof The group $S$ is normalised by both $D$ and $X$, so $G \cap T^n = S$.  Let $\nu_1$ denote the projection from $(H \wr X)_1$ to $H$.  Since $\langle S, D \rangle$ fixes the coordinate $1$, both $\nu_1(S)=T$ and $\nu_1(D)=H_0$ are contained in $\nu_1(G_1)$; that is $H = \langle T,H_0 \rangle \leq \nu_1(G_1)$, so the component of $G$ is $H$.

The stabiliser in $G$ of the vertex $(1,0) \in B_1$ fixes the block $B_1$, and induces $H_0$ on $B_1$ and $X_1$ on the set of blocks. So $G_{(1,0)}$ is transitive on $B_1 \backslash \{(1,0)\}$ and on ${\mathcal B} \backslash \{B_1\}$.  Moreover, the stabiliser in $G_{(1,0)}$ of $B_2$ contains $S_{(1,0)} = \{(1,t_2,\ldots,t_n) \, | \, t_2 t_3 \ldots t_n = 1\}$, and since $n > 2$ we have $\pi_2(S_{(1,0)}) = T$; so $(G_{(1,0)})_{B_2}$ induces a transitive action on $B_2$.  It follows that $G$ has rank $3$. \qed\medskip

%%%%%%%%%%%%%%%%%%%%%%%%%%%%%%%%%%%%%%%%%%%%%%%%%%%%%%%%%%%%%%%%%%%%%%%%%%%%%%%%%%%%%%%%%%%%%%%%%%%%%%%%%%%%%%%%%%%%%%%%%%%%%%%%%

\section{Proofs of Theorem \ref{mainthm} and \ref{trivkerthm}} \label{trivkerthmproof}

In this section we will suppose that $G\leqslant \Sym(\Omega)$ satisfies Hypothesis \ref{mainhypo} with corresponding block-set $\calB=\{B\times\{i\}|1\leq i\leq n\}$, that $G$ is block-faithful and $G\cong G^\calB$ is almost simple.  
By Lemma \ref{cc} and Corollary \ref{AS}, this is in particular the case if the action of $G$ on $\Omega$ has  rank 3 and $G$ is block-faithful.

Therefore, we will prove Theorem \ref{mainthm}  and Theorem  \ref{trivkerthm} in parallel. 
In particular both results will follow from Theorem \ref{2transgroupsAS} and the propositions in this section.
Our method is to first determine the 2-transitive almost simple groups whose   
 point-stabiliser admits a (not necessarily 
faithful) 2-transitive action. We then determine which of these give rise to a rank 3 action.
We will often need to use two blocks of $\calB$ which, for simplicity of notation, we will simply denote by $B$ and $B'$.

For convenience we have listed the almost simple 2-transitive groups in the following theorem (compiled from \cite[Table 7.4]{CameronBook}).

\begin{theorem} \label{2transgroupsAS}
Let $G$ be a finite almost simple $2$-transitive permutation group on $n$ points, with unique minimal normal subgroup $T$.
Then one of the following holds.
\begin{itemize}
\item[(a)] $G = \PSL(2,11)$, of degree $11$; $G = \PGammaL(2,8)$, of degree $28$; $G = A_7$,
of degree $15$; $G =\HS$, of degree $176$; $G = \Co_3$, of degree $276$.
\item[(b)] $T = M_n$, of degree $n = 11,12,22,23$ or $24$; or $G = M_{11}$ of degree $12$.
\item[(c)] $T = A_n$, of degree $n \geq 5$.
\item[(d)] $T = \PSp(2\ell,2)$ of degree $n=2^{2\ell-1} \pm 2^{\ell-1}$, with $\ell \geq 3$.
\item[(e)] $T=\Sz(q)$, of degree $n=q^2+1$, with $q = 2^{2c+1} \geq 8$.
\item[(f)] $T=\Ree(q)$, of degree $n=q^3 + 1$, with $q = 3^{2c+1} > 3$.
\item[(g)] $T = \PSU(3,q)$, of degree $n=q^3 + 1$ with $q \geq 3$.
\item[(h)] $T = \PSL(a,q)$ of degree $n=(q^a-1)/(q-1)$ with $(a,q) \neq (2,2)$ or $(2,3)$.
\end{itemize}
\end{theorem}

Two actions of the group $G$ are said to be {\it isomorphic} if the point stabilisers of the two actions
are conjugate in $G$.

\begin{remark}{\rm
In each of (a)-(h), there are at most two 2-transitive actions up to isomorphism. Moreover, whenever one of these groups, say $G$, has two nonisomorphic 2-transitive actions, the corresponding point stabilisers are interchanged by an outer automorphism of $G$ (see the final column of  \cite[p.197]{CameronBook}). Therefore it is sufficient to consider one action for the 2-transitive almost simple group acting faithfully on blocks.
}\end{remark}

The following lemma will be useful.
\begin{lemma} \label{properties}
Suppose that $G\leqslant \Sym(\Omega)$ is a group satisfying Hypothesis $\ref{mainhypo}$, with system of imprimitivity ${\mathcal B}$ consisting of $n$ blocks of size
$m$. 
Then $G$ has rank $3$ if and only if 
\begin{itemize}
 \item[(A)] for distinct $B,B'\in {\mathcal B}$, $G_{B,B'}$ acts transitively on $B \times B'$, and
 \item[(B)] for $B\in {\mathcal B}$ and $v\in B$, the subgroup $G_v=G_{B,v}$ induces a transitive
 action on ${\mathcal B}\setminus\{B\}$.
 \end{itemize}
If $G$ has rank 3, then for distinct $B,B'\in {\mathcal B}$, the following properties hold:
\begin{itemize}
 \item[(C)] the number $m^2$ divides $|G_{B,B'}|$;
 \item[(D)] the subgroup $G_{B,B'}$ admits at least two non-isomorphic transitive actions
 on $m$ points;
 \item[(E)] for $B\in {\mathcal B}$, the action of $G_B$ on $B$ is not isomophic to the
 action of $G_B$ on ${\mathcal B}\setminus\{B\}$;
\item[(F)] for $B\in {\mathcal B}$ and $v\in B$, $G_v$ is not $G_B\cap N$ for any normal subgroup $N$ of $G$.
\end{itemize}
\end{lemma}
\begin{proof}
The group $G$ has rank 3 if and only if the orbits of the stabiliser $G_v$ are $\{v\}$, $B\setminus \{v\}$ and $\Omega\setminus B$, for any $v\in \Omega$, with $B$  the block of imprimitivity containing $v$.

Suppose first that properties (A) and (B) hold. 
Since $G_B^B$ is 2-transitive by Hypothesis \ref{mainhypo},  $B\setminus \{v\}$ is an orbit of $G_v$. Let $v_1,v_2\in \Omega\setminus B$ be distinct and let $B_1$, respectively $B_2$, be the block containing $v_1$, respectively $v_2$. By (B), there exists $g\in G_v$ mapping  $B_1$ onto $B_2$, and by (A), there exists $h\in G_{B,B_2}$ mapping $(v,v_1^g)$ onto $(v,v_2)$. Note that $h\in G_v$ and $gh\in G_v$ maps $v_1$ onto $v_2$. Thus  $\Omega\setminus B$ is also an orbit of $G_v$ and $G$ has rank 3.

Conversely, suppose that $G$ has rank 3. 
Let $B$, $B'$ be distinct blocks of $\calB$, and let $v\in B$.
Let $(v_1,v'_1)$ and $(v_2,v'_2)$ be two elements of $B \times B'$. 
Then there exists $g_1\in G_{v_1}$ mapping $v'_1$ onto $v'_2$, since they both belong to $\Omega\setminus B$, and similarly, there exists $g_2\in G_{v'_2}$ mapping $v_1$ onto $v_2$. 
Thus $g_1g_2$ maps  $(v_1,v'_1)$ onto $(v_2,v'_2)$ and maps the block $B$ containing $v_1$ to the block $B$ containing $v_2$, that is $g_1g_2$ fixes $B$ and similarly $g_1g_2$ fixes $B'$. Thus $g_1g_2\in G_{B,B'}$. Therefore property (A) holds.
Let $B\in {\mathcal B}$ and $v\in B$. Since $G_v$ is transitive on  $\Omega\setminus B$, property (B) follows.
 
Property (C) follows immediately from property (A). It also follows from (A) that $G_{B,B'}$ admits at least one transitive action on $m$ points. If $G_{B,B'}$ admitted only one transitive action on $m$ points (up to conjugation in $G$), then the stabiliser in $G_{B,B'}$ of a vertex in $B$ would also fix a
vertex in $B'$, which would contradict (A). Hence property (D) holds.
If $G_B^B$ was isomorphic to the
action of $G_B$ on ${\mathcal B}\setminus\{B\}$, then $G_{B,B'}$ would fix a point of $B$, contradicting (A). Hence property (E) holds.

Suppose $G_v=G_B\cap N$ where $N$ is a normal
subgroup of $G$. Since $G^{\mathcal B}$ is
2-transitive, there exists $g\in G$ interchanging $B$ and
$B'$. Let $v'=v^g$. Then $G_{v'}= G_v^g=G_B^g\cap N^g=G_{B'}\cap N$. But then
$G_{B,B',v}=G_B\cap N\cap G_{B'}$ is contained in $G_{v'}$ and hence fixes the point $v'$ of
$B'$, contradicting property (A). Hence (F) holds.
\end{proof}

We use the following lemma in the proofs of several Propositions in this section.
\begin{lemma}\label{GBBaffine}
Suppose $G$ satisfies Hypothesis \ref{mainhypo} and let $B$ be a block of imprimitivity. If $G_B^B$ is $2$-transitive of affine type, then $G_B/G_{(B)}\cong G_B^B$ has a unique minimal normal subgroup $K /G_{(B)}$ which is self-centralising in $G_B/G_{(B)}$. 
\end{lemma}
\begin{proof}
 Since $G_B/G_{(B)}\cong G_B^B$ is $2$-transitive of affine type, it has a unique minimal normal subgroup $N$ and $N$ is self-centralising in $G_B^B$. Since $N=K /G_{(B)}$ for some normal subgroup $K$ of $G_B$, the result follows.
\end{proof}

We deal with the various cases of Theorem \ref{2transgroupsAS} in a sequence of propositions.
\begin{prop}
Let $G$ be an imprimitive permutation group on $\Omega$ with system of imprimitivity $\calB$ such that  $G\cong G^\calB$  is as in (a) or (b) of Theorem $\ref{2transgroupsAS}$.
Then $G$ satisfies Hypothesis $\ref{mainhypo}$ if and only if
$G$, $n$ and $|B|$ are as in one of Lines $1$ to $16$ of Table $\ref{bigtable1}$. 
Moreover $G$ has rank $3$ on $\Omega$ if and only if $G=M_{11}$, $n=11$, $|B|=2$ and $G_B^B\cong C_2$.
\end{prop}

\begin{proof}
Lines 1 to 16 of Table \ref{bigtable1} contain in column 2 all the 2-transitive groups $G^\calB$ in parts (a) or
(b) of Theorem \ref{2transgroupsAS} together with the degree $|\calB|=n$, the stabiliser $(G^\calB)_B\cong G_B$ and
the stabiliser of 2 distinct blocks in columns 3, 5 and 6 respectively. Column 4 contains
the degrees $|B|$ of all 2-transitive representations of $G_B$, using information derived from
Theorem \ref{2transgroupsAS}. 
Thus $G\cong G^\calB$ satisfies Hypothesis \ref{mainhypo} with respect to $n=|\calB|$ and some $B$ if and only if  $G$, $n$, $|B|$ occur in Lines 1 to 16 of Table \ref{bigtable1}.

Now assume that $G$ has rank 3. By Lemma \ref{cc}, $G$ is in one of the lines 1 to 16 of Table \ref{bigtable1}.
Cases on Lines 1, 2, 3, 4, 7, 9, 10, 12, 13, 15, 16 do not satisfy the divisibility condition (C). Hence $G$ is in one of the other cases.

Suppose $G$ were as in  Line 5. We have
$G_B=\PSigmaU(3,5)$ and $G_{(B)}=\PSU(3,5)$.  If $B'$ is a second block of imprimitivity,
then $G_{B,B'}=\Aut(A_6)$ and it can be deduced from the Atlas \cite[p34]{atlas} that
$G_{(B),B'}=M_{10}$. Notice that $G_{(B),B'}$ is contained in $G_{B'}\cong \PSigmaU(3,5)$
and that $G_{(B')}$ is the unique subgroup of $G_{B'}$ isomorphic to $\PSU(3,5)$. By
\cite[p34]{atlas}, the group $\PSigmaU(3,5)$ contains two classes of $M_{10}$ subgroups,
all of which are contained in $\PSU(3,5)$. It follows that $G_{(B),B'}$ is contained in
$G_{(B')}$, and so (A) is not satisfied.

Suppose $G$ were as in  Line 6. Then $G_{B,B'}=\PSU(4,3)\rtimes C_2$ has only
one conjugacy class of subgroups of index 2, namely subgroups isomorphic to $\PSU(4,3)$. This contradicts (D).

Suppose $G$ were as in  Line 8. Then $G_B^B=C_2$ and $G_v=G_{(B)}=A_6$, and it can be checked (for instance with Magma \cite{magma}) that this $G$-action has rank 3.

Suppose $G$ were as in  Line 11. Then $G_{B,B'}=M_{10}$ has only one conjugacy class of
subgroups of index 12, namely subgroups isomorphic to $A_5$. This can be deduced from the
Atlas \cite[p.4]{atlas}. This contradicts (D).

Suppose $G$ were as in  Line 14. For $v\in B$, we have
$G_{B,v}=\PSL(3,4)=\PSigmaL(3,4)\cap M_{22}$, contradicting (F).
\end{proof}

\begin{prop}
Let $G$ be an imprimitive permutation group on $\Omega$ with system of imprimitivity $\calB$ such that  $G\cong G^\calB$  is as in (c) of Theorem $\ref{2transgroupsAS}$.
Then $G$ satisfies Hypothesis $\ref{mainhypo}$ if and only if
$G$, $n$ and  $|B|$ are as in one of Lines $17$ to $26$ of Table $\ref{bigtable1}$. 
Moreover, none of these group actions has rank $3$.
\end{prop}
\proof
We have $\soc(G) = A_n$ with $n\geq 5$, so that $G=S_n$ or $A_n$.  Then $G_B =
S_{n-1}$ or $A_{n-1}$, and $G_{B,B'}= S_{n-2}$ or $A_{n-2}$ respectively. 
Lines $17$ to $26$ of Table \ref{bigtable1} list the possible degrees $m=|B|$ of a 2-transitive action for $G_B$.
Now assume the action of $G$ on $\Omega$ has rank 3. Then $G$ is as in one of Lines $17$ to $26$ of Table \ref{bigtable1}.
 By property (E), $G^B_B$ is not the natural action of $G_B$ on $n-1$ points. So if $m=n-1$, then
$n=7$, we are in the case of Line 17 or 18, and $G^B_B$ is not the natural action of $S_6$ or $A_6$ on $6$ points. However in those cases $G_{B,B'}$ is $A_5$ or $S_5$ and does not have property (C).  Next
if $G=S_n$ and $m=2$, then $G_{B,v}=A_{n-1}=G_B\cap A_n$, contradicting property (F), so Line 19 does not occur.
This leaves only the possibilities for $G_B$ and $m$ in Lines 20 to 26 of  Table \ref{bigtable1}.  None of these satisfy (C). So such a group $G$ does not exist.\qed\medskip

\begin{prop}
Let $G$ be an imprimitive permutation group on $\Omega$ with system of imprimitivity $\calB$ such that  $G\cong G^\calB$  is as in (d) of Theorem $\ref{2transgroupsAS}$.
Then $G$ satisfies Hypothesis $\ref{mainhypo}$ if and only if
 $G$, $n$  and $|B|$ are as in one of Lines $27$ to $29$ of Table $\ref{bigtable1}$. 
Moreover, none of these group actions has rank $3$.
\end{prop}
\proof 
Here $G=\PSp(2\ell,2)$, of degree $n=2^{2\ell-1} \pm 2^{\ell-1}$, with $\ell \geq 3$. We have $G_B=\PSO^{\pm}(2\ell,2)$. Let $B'$ be another block in $\calB$. Then $G_{B,B'}$ is the stabiliser of a singular point in the action of $\PSO^{\pm}(2\ell,2)$ on $V(2\ell,2)$, that is, $G_{B,B'}\cong N\rtimes \PSO^{\pm}(2\ell-2,2)$, with $N$ elementary abelian
of order $2^{2\ell-2}$.

Assume first that either $\ell\geq 4$, or $\ell=3$ and $G^\calB$ has degree $28$. Then there is no 2-transitive action of $G_B$ with socle
$\POmega^{\pm}(2\ell,2)$, and so the only 2-transitive action of $G_B$ is on 2 points, with $G_{(B)}=\POmega^{\pm}(2\ell,2)$,  as in  Lines 27 and 28 (with $\ell\geq 4$) of Table \ref{bigtable1}.In this case, $G_{(B),B'}$ is an index 2 subgroup of $G_{B,B'}$. As the induced action  of $G_{B,B'}$ on $N$ is irreducible, we must have $G_{(B),B'}=N\rtimes H$ where $H$ is an index 2 subgroup of $\PSO^{\pm}(2\ell-2,2)$. The only such subgroup is  $\POmega^{\pm}(2\ell-2,2)$, since $\POmega^{\pm}(2\ell-2,2)$ is simple. Therefore $G_{(B),B'}=N\rtimes \POmega^{\pm}(2\ell-2,2)$. Now $G_{(B),B'}$ is contained in $G_{B'}\cong \PSO^{\pm}(2\ell,2)$.
If $G_{(B),B'}$ is not contained in $G_{(B')}%=G'_{B'}
\cong \POmega^{\pm}(2\ell,2)$, then  $G_{(B),B'}\cap G_{(B')}$ has index 2 in $G_{(B),B'}$. Since  $\POmega^{\pm}(2\ell-2,2)$ is simple and acts irreducibly on  $N$, this is not possible. Hence $G_{(B),B'}$ is contained in $G_{(B')}$, and so property (A) fails. Thus this action of $G$ does not have rank 3.

Now assume  $\ell=3$ and $G$ has degree $36$. Then $G_B=\PSO^{+}(6,2)\cong S_8$, which has 2-transitive actions on 2 and 8 points, as in  Lines 28 (with $\ell=3$) and 29 of Table \ref{bigtable1}.
The only index 35 subgroups of $S_8$ are stabilisers of  partitions of an 8-set into
two 4-sets (see the Atlas
\cite[p.22]{atlas}), and hence $G_{B,B'}\cong S_4\wr S_2$.

Consider first the case $|B|=2$. Then $G_{(B)}=\POmega^{+}(6,2)\cong A_8$, and by \cite[p.22]{atlas} $G_{(B),B'}\cong (S_4\wr S_2)\cap A_8%=(2^4\rtimes(C_3\times C_3)).2^2
\cong 2^4\rtimes(S_3\times S_3)$ acting transitively on 8 points. Now $G_{(B),B'}$ and $G_{B,(B')}$ have index 2 in $G_{B,B'}\cong S_4\wr S_2$. 
Any index 2 subgroup in $G_{B,B'}$ contains the derived subgroup  $(G_{B,B'})'=(2^4\rtimes(C_3\times C_3)).2$, and so there are exactly three subgroups of index 2 in $G_{B,B'}$, namely $S_4\times S_4$, $(S_4\wr S_2)\cap A_8$ and $(2^4\rtimes(C_3\times C_3)).C_4$, and they are pairwise nonisomorphic. 
Since $G_{(B),B'}$ and $G_{B,(B')}$ are conjugate in $G$, we must have that $G_{(B),B'}=G_{B,(B')}$. Thus property (A) fails and this action of $G$ does not have rank 3.

Now consider the case $|B|=8$. Here $G_{B,v}\cong S_7$ for $v$ a point of $B$.
Therefore
$G_{B,B',v}$ is the stabiliser in $G_B=S_8$ of a partition with parts of size $1$, $3$, $4$, and $G_{B,B',v}=S_4\times S_3$. We have $G_{B,B',v}<G_{B'}\cong S_8$. 
It can be computed, for instance with Magma \cite{magma}, that there is only one
conjugacy class of subgroups isomorphic to $S_4\times S_3$ in $S_8$, and each such group has orbits of size  $1$, $3$, $4$.  Hence $G_{B,B',v}$ fixes a point of $B'$, contradicting property (A).
So this action of $G$ does not have rank 3 either.
\qed\medskip

\begin{prop}
Let $G$ be an imprimitive permutation group on $\Omega$ with system of imprimitivity $\calB$ such that  $G\cong G^\calB$  is as in (e) of Theorem $\ref{2transgroupsAS}$.
Then $G$ satisfies Hypothesis $\ref{mainhypo}$ if and only if $G$, $n$ and $|B|$ are as in Line $1$ of Table $\ref{bigtable2}$. 
Moreover, this group action does not have rank $3$.
\end{prop}
\proof
We have $T=\Sz(q)$ with $q=2^e$ and $e$ odd, $e\geq 3$, $T\leq G\leq \Aut(T)=A$, and $A_B=(Q\rtimes\langle \tau\rangle)\rtimes\langle\sigma\rangle$, where
$|Q|=q^2$, $|\tau|=q-1$ and $|\sigma|=e$. More precisely, $G=T\rtimes\langle \sigma^{e/d}\rangle$ for some divisor $d$ of $e$. 
Now $T_B$ has a unique minimal normal
subgroup $K$. Moreover, $|K|=q$, $K$ is the centre of $Q$ and $A_B/K\cong\AGammaL(1,q)$, see
\cite{suz}.  Furthermore, $G_B=(Q\rtimes\langle \tau\rangle)\rtimes\langle
\sigma^{e/d}\rangle$ and we can choose $B'$ such that $G_{B,B'}=\langle
\tau\rangle\rtimes\langle \sigma^{e/d}\rangle$.

Now $G_B=A_B\cap G$, which contains $T_B=Q\rtimes\langle \tau\rangle$.
Suppose first that $G_B$ acts faithfully on $B$. Then $G_B^B$ is affine with minimal normal subgroup $K$ and the fact that $Q$
centralises $K$ contradicts Lemma \ref{GBBaffine}. Thus $G_B$ is unfaithful on $B$ and so $K^B=1$. Moreover, as a
2-transitive group has even order and both $q-1$ and $e$ are odd, it follows that
$G_{(B)}=K$. Then $\AGL(1,q)\leq G_B^B\leq \AGammaL(1,q)$ and the only faithful
2-transitive action of this group is of degree $q$. Thus $|B|=q$, as in  Line 1 of Table \ref{bigtable2}. 
If this $G$-action had rank 3, then by property (C), we would have $m^2=q^2$ dividing $|G_{B,B'}|$, which is not the case.
\qed\medskip

\begin{prop}
Let $G$ be an imprimitive permutation group on $\Omega$ with system of imprimitivity $\calB$ such that  $G\cong G^\calB$   is as in (f) of Theorem $\ref{2transgroupsAS}$.
Then $G$ satisfies Hypothesis $\ref{mainhypo}$ if and only if
$G$, $n$, $|B|$ are as in  Lines $2$ or $3$ of Table $\ref{bigtable2}$. 
Moreover, neither of these group actions has rank $3$.
\end{prop}
\proof
We have $T=\Ree(q)$ with $q=3^e$ and $e$ odd, $e\geq 3$, $T\leq G\leq\Aut(T)=A$ and  $A_B=(Q\rtimes\langle \tau\rangle)\rtimes\langle \sigma\rangle$, where
$|Q|=q^3$, $|\tau|=q-1$ and $|\sigma|=e$. Moreover, $|Z(Q)|=q$, $|Q'|=q^2$ and $\langle
\tau\rangle$ acts irreducibly on the elementary abelian 3-groups $Z(Q)$, $Q'/Z(Q)$ and
$Q/Q'$. Furthermore, $\tau$ acts transitively on the nontrivial elements of $Z(Q)$ and
$Q/Q'$, while having two equal length orbits on the nontrivial elements of $Q/Z(Q)$. See
\cite{KLM,Ree,ward} for further details. We can choose another block $B'$ such that
$G_{B,B'}= (\langle\tau\rangle\rtimes\langle \sigma\rangle)\cap G$.

Now $G=T\rtimes\langle \sigma^{e/d}\rangle$ for some divisor $d$ of $e$ and $G_B=A_B\cap
G$, which contains $T_B=Q\rtimes\langle \tau\rangle$. Given the way that $\tau$ acts on
$Q$, it follows that the nontrivial normal subgroups of $G_B$ contained in $Q$ are
$Z(Q),Q'$ and $Q$.

Suppose first that $G_B$ acts faithfully on $B$.
 Then $G_B^B$ is affine with minimal normal subgroup $Z(Q)$ and the fact that $Q$
centralises $Z(Q)$ contradicts Lemma \ref{GBBaffine}.   Thus $Z(Q)\leq G_{(B)}$. If $Z(Q)=G_{(B)}$ then $Q'/Z(Q)$ is the
unique minimal normal subgroup of $G_B^B$. However, $Q'/Z(Q)$ is central in $Q/Z(Q)$ which is a contradiction, 
so $Q'\leq G_{(B)}$. If $G_{(B)}=Q'$ then $G_B^B\leq \AGammaL(1,q)$ and contains
$\AGL(1,q)$. In this case, the only faithful 2-transitive action of $G_B^B$ is of degree $q$. Thus
$|B|=q$, as in  Line 2 of Table \ref{bigtable2}. 
If this $G$-action had rank 3, then by property (C), we would have $q^2$ dividing $|G_{B,B'}|$, which is not the case.
 If $G_{(B)}\neq Q'$ then $Q\leq G_{(B)}$. Since $e$ is odd and the order of the 2-transitive group $G_B^B$ is even, it follows that $\langle\tau\rangle$
induces a nontrivial cyclic group on $B$; and as $G_B^B$ is 2-transitive, the degree
$m=|B|$ is a prime dividing $q-1$ such that $m-1$ divides $e$. Hence $|B|=2$, as in  Line 3 of Table \ref{bigtable2}, and
$G_B^B=C_2$. 
 If this $G$-action had rank 3, then by property (D), $G_{B,B'}$  would have at least two non-isomorphic actions on 2 points, which is not the case.
\qed\medskip

We next consider the groups with socle $\PSU(3,q)$
\begin{remark}[\cite{ONan}]\label{PSUdescr}{\rm
 Suppose that $T=\PSU(3,q)\leq G\leq \PGammaU(3,q)=A$, with $A$ acting on $\calB$ of size $q^3+1$, where $q=p^e$ with $p$ a prime and $e\geq 1$, $q>2$ and let $B\in\calB$. The stabiliser $A_B$ has a normal subgroup $Q$ of order $q^3$ and contains
elements $\tau,\sigma$ of order $q^2-1$ and $2e$, respectively, such that $T_B=Q\rtimes
\langle \tau^{(3,q+1)}\rangle\leq G_B\leq (Q\rtimes \langle \tau\rangle)\rtimes\langle
\sigma\rangle=A_B$. More precisely, $G_B=A_B\cap G$. Moreover, $T_B$ has a unique minimal
normal subgroup $K$, which is contained in and centralised by $Q$, $|K|=q$, $Q/K$ is elementary abelian of
order $q^2$ and $A_B/K\cong\AGammaL(1,q^2)$.}
\end{remark}

\begin{prop}\label{prop:PSU}
Let $G$ be an imprimitive permutation group on $\Omega$ with system of imprimitivity $\calB$ such that  $G\cong G^\calB$   is as in (g) of Theorem $\ref{2transgroupsAS}$.
Then $G$ satisfies Hypothesis $\ref{mainhypo}$ if and only if
 $G$, $n$ and $|B|$ are as in Lines $4$, $5$, or $6$ of Table $\ref{bigtable2}$, with column $5$ giving additional conditions on $G$ or $|B|$.
Moreover, none of these group actions has rank $3$.
\end{prop}
\proof
Let $T=\PSU(3,q)\leq G\leq \PGammaU(3,q)=A$ and $B$ be a block of $\calB$. We use the notation of Remark \ref{PSUdescr}. 
 We can choose another
block $B'$ such that $G_{B,B'}= (\langle\tau\rangle\rtimes\langle \sigma\rangle)\cap
G$. Moreover, there is an element $g\in \PSU(3,q)$ interchanging $B$ and $B'$ such that $\tau^g=\tau^{-q}$ and
$\sigma^g=\sigma$. Indeed, using the notation of \cite[p.249]{DM}, with $B=\langle e_1\rangle$, $B'=\langle e_3\rangle$, then $\tau=h_{\gamma,1}$ for some primitive element $\gamma$ of $\GF(q^2)$ and
 $g$ is the map $(x_1,x_2,x_3)\mapsto (x_3,-x_2,x_1)$. It is straightforward to check that $g$ preserves the hermitian form $\phi$ in \cite[p.249]{DM}, that $\Det(g)=1$ and $g^2=1$.

Suppose first that $G_B$ acts faithfully on $B$. Then $G_B^B$ is affine with minimal normal subgroup $K$ and the fact that $Q$
centralises $K$ contradicts Lemma \ref{GBBaffine}. 
 Hence $G_B$ is
unfaithful on $B$ and $K^B=1$. Suppose first that $G_B^B=G_B/K\leq\AGammaL(1,q^2)$. Then
$|B|=q^2$ since the only faithful primitive action of this group is on $q^2$ points. This action is $2$-transitive provided that $G\cap(\langle \tau\rangle\rtimes\langle \sigma\rangle)$ is a
subgroup of $\GammaL(1,q^2)$ which acts transitively on the set of nonzero elements of
$\GF(q^2)$. Thus Line 4 of Table \ref{bigtable2} holds. If this $G$-action had rank 3, then by property (C), we would have $|B|^2=q^4$ dividing $|G_{B,B'}|$, which is not the case. 
Suppose now that 
$K<G_{(B)}$. Since $q>2$, $G_B$ acts irreducibly on $Q/K$  and $Q/K$ is the unique minimal normal subgroup of $G_B/K$. 
Thus $Q\leq G_{(B)}$. Since
$G_B^B$ is 2-transitive, it follows that either
   \begin{enumerate}
     \item $|B|=m$ is an odd prime dividing $q^2-1$ such that the order of $p\pmod m$, denoted by $\ord(p \mod m)$, is $m-1$, and
$G_B^B\cong C_m\rtimes C_{m-1}\cong \AGL(1,m)$, or
     \item $m=2$, $|G_B/Q|$ is even and $G_B^B=C_2$.
    \end{enumerate}

Assume we are in case (1). Thus Line 5 of Table \ref{bigtable2} holds. Suppose that this $G$-action has rank 3.
Then $G_{B,B'}= (\langle\tau\rangle\rtimes\langle \sigma\rangle)\cap G$ has a
unique transitive representation of degree $m$, contradicting property (D).

Assume we are in case (2). Let $H=G\cap\PGU(3,q)$.
 Since  $|G_B/Q|$ is even, we either have that $q$ is odd, or that $|G/H|$ is even, or both. 
 Thus Line 6 of Table \ref{bigtable2} holds. Suppose that this $G$-action has rank 3.
 If either $q$ is even or $|G/H|$ is odd, then a Sylow $2$-subgroup of $G_{B,B'}$ is cyclic and hence $G_{B,B'}$ has at most one transitive representation on 2 points, so property (D) is not satisfied. Thus $q$ is odd and $|G/H|$ is even, and so $|G/H|=2e/i$ for some divisor $i$ of $e$. Then $G_B$, $G_{B,B'}$, $G_{(B)}$ and $G_{(B),B'}$ are as in one of the lines of Table \ref{PSU-2}. In lines (1.b), (1.c), (2.b), (2.c), there are additional conditions on $p$ and $i$  for $G_{(B)}$ to have index 2 in $G_B$ (but we shall not need them here).

\begin{table}
\begin{center}
\begin{tabular}{|l|l|l|l|l|}
\hline 
$G_B$& $G_{B,B'}$&$G_{(B)}$&$G_{(B),B'}$&\\
\hline
$ Q\rtimes\la\tau^3,\mu\ra$& $\la\tau^3,\mu\ra$&$ Q\rtimes\la\tau^3,\mu^{2}\ra$& $\la\tau^3,\mu^{2}\ra$&(1.a)\\
where $\mu=\sigma^i$, $\sigma^i\tau$ or $\sigma^i\tau^2$& &$ Q\rtimes\la\tau^6,\mu\ra$& $\la\tau^6,\mu\ra$&(1.b)\\
and $3$ divides $q+1$ &&$ Q\rtimes\la\tau^6,\mu\tau^3\ra$& $\la\tau^6,\mu\tau^3\ra$&(1.c)\\
\hline
$(Q\rtimes\la\tau\ra)\rtimes\la\sigma^i\ra$&$\la\tau\ra \rtimes \la\sigma^i\ra$ &$(Q\rtimes\la\tau\ra)\rtimes\la\sigma^{2i}\ra$&  $\la\tau\ra\rtimes\la\sigma^{2i}\ra$&(2.a)\\
&& $(Q\rtimes\la\tau^2\ra)\rtimes\la\sigma^i\ra$& $\la\tau^2\ra\rtimes\la\sigma^{i}\ra$&(2.b)\\
&& $ Q\rtimes\la\tau^2,\sigma^i\tau\ra$& $\la\tau^2,\sigma^{i}\tau\ra$&(2.c)\\
\hline
\end{tabular}\\
\caption{Possibilities for the proof of Proposition \ref{prop:PSU}} \label{PSU-2}
\end{center}
\end{table}

If $G= \langle H,\mu\rangle$, where $\mu=\sigma^i, \sigma^i\tau$ or
$\sigma^i\tau^2$, for some $i$ dividing $e$, then $\langle H,\mu^2
\rangle$ is normal in $G$, and so, by property (F), $G_{(B)}$ cannot be equal to $\langle H,\mu^2
\rangle\cap
G_B$. This eliminates cases (1.a) and (2.a).

In all other cases $G_{(B),B'}=\langle\tau^{2\ell},\sigma^{i}\tau^k\rangle$ for some $k\in [0,2\ell-1]$ where $\ell=3$ in cases (1.b) and (1.c), and $\ell=1$ in cases (2.b) and (2.c).
Recall that $g$ interchanges $B$ and $B'$, $\tau^g=\tau^{-q}$ and $\sigma^g=\sigma$.  Thus  $G_{(B'),B}= (G_{(B),B'})^g$ is generated by $(\tau^{2\ell})^g=\tau^{-2\ell q}$ and $(\sigma^{i}\tau^k)^g=\sigma^i\tau^{-kq}$. Since $\tau^{-2\ell q}\in\langle\tau^{2\ell}\rangle$ and $\tau^{2\ell}=(\tau^{-2\ell q})^{-q}\in\langle\tau^{-2\ell q}\rangle$, we have  $\langle\tau^{2\ell}\rangle=\langle(\tau^{2\ell})^g\rangle$. Moreover, $\sigma^i\tau^{-kq}=(\sigma^i\tau^{k})\tau^{-k(q+1)}\equiv \sigma^i\tau^{k}$ modulo $ \langle\tau^{2\ell}\rangle$ (since $2\ell$ divides $q+1$ in both cases). It follows that  $G_{(B'),B}= G_{(B),B'}$ and hence $G_{(B),B'}$ fixes $B\cup B'$ pointwise and property (A) fails. Thus this $G$-action does not have rank 3.
\qed\medskip

We next consider the groups with socle $\PSL(2,q)$. 
\begin{remark}\label{PSLdescr}{\rm
Let $T=\PSL(2,q)\leq G\leq \PGammaL(2,q)=A$, with $A$ acting on $\calB$ of size $q+1$, where $q=p^e\geq 4$ with $p$ a prime and $e\geq 1$, and let $B\in\calB$. Then $G_B=((Q\rtimes\la \tau\ra)\rtimes\la \sigma\ra)\cap G$ where $Q$ is elementary abelian of order $q$,  $\tau$ is of order $q-1$ and $\sigma$ is of order $e$. Let $i=|G/(G\cap \PGL(2,q))|$, then $i$ divides $e$.}
\end{remark}

\begin{prop}\label{PSL2}
Let $G$ be an imprimitive permutation group on $\Omega$ with system of imprimitivity $\calB$ such that  $G\cong G^\calB$   is as in (h) of Theorem $\ref{2transgroupsAS}$ for $a=2$.
Then $G$ satisfies Hypothesis $\ref{mainhypo}$ if and only if
 $G$, $n$ and $|B|$ are as in Lines $7$, $8$, or $9$ of Table $\ref{bigtable2}$, with column $5$ giving additional conditions on $G$ or $|B|$. 
Moreover, $G$ has rank $3$ on $\Omega$ if and only if the following conditions hold (see Remark \ref{PSLdescr} for the notation):
\begin{itemize}
\item $|B|=2$, $q\equiv 1 \pmod 4$,
\item  $G=\la\PSL(2,q),\sigma^i\tau\ra$, where $i$ divides  $e$, $e/i$ is even, and either $p^i\equiv 3\pmod 4$, or $p^i\equiv 1\pmod 4$ and $e/i\equiv 0 \pmod 4$,
\item $G_{(B)}= Q\rtimes\la\tau^4,\sigma^i\tau^k\ra$ with $k=1$ or $3$.
\end{itemize}
Moreover, the two actions (for $k=1$ or $3$) are not isomorphic.
\end{prop}
\proof
Let $T=\PSL(2,q)\leq G\leq \PGammaL(2,q)=A$ and $B$ be a block of $\calB$. We use the notation of Remark \ref{PSLdescr}. 
We can choose another block $B'$ such that $G_{B,B'}= (\la\tau\ra\rtimes\la \sigma\ra)\cap G$. Also there is an element $g\in \PSL(2,p)\leq \PSL(2,q)$ interchanging $B$ and $B'$ and such that $\sigma^g=\sigma$ and $\tau^g=\tau^{-1}$. Indeed, working in $\GammaL(2,q)$ acting on the vector space $V(2,q)=\langle e_1,e_2\rangle$, we may take $B=\langle e_1\rangle$, $B'=\langle e_2\rangle$, $\tau:(x_1,x_2)\mapsto (x_1,\gamma x_2)$ for some primitive element $\gamma$ of $\GF(q)$, and $g$ to be the map $(x_1,x_2)\mapsto (-x_2,x_1)$.

If $G_B$ acts faithfully on $B$ then $Q$ is regular on $B$ and $|B|=q$. 
Since $G_B^B$ is 2-transitive it follows that  $G\cap(\la \tau\ra\rtimes\la \sigma\ra)$ is a subgroup of
$\GammaL(1,q)$ which acts transitively on the set of nonzero elements of
$\GF(q)$, in other words $G$ is 3-transitive, and Line 7 of Table \ref{bigtable2} holds. 
 If this $G$-action had rank 3, then by property (C), we would have $|B|^2=q^2$ dividing $|G_{B,B'}|$, which is not the case.

Assume now that $G_B$ acts unfaithfully on $B$. Then, as $Q$ is the unique minimal normal subgroup of $G_B$,  $Q\leq G_{(B)}$. 
Thus either $|B|=m$, where $m$ is an odd prime dividing $q-1$ such that the order of $p \pmod m$ is $m-1$, and $G_B^B\cong C_m\rtimes C_{m-1}\cong \AGL(1,m)$, or $|B|=2$, $|G_B/Q|$ is even and $G_B^B=C_2$. Note that $|G_B/Q|$ even holds if and only if either $q\equiv 1 \pmod 4$, or  $q\equiv 3\pmod 4$ and $G\geq \PGL(2,q)$, or $|G/(G\cap \PGL(2,q))|$ is even. Thus  so Line 8 or 9 of Table \ref{bigtable2} holds. 

Assume we are in the first case and that the $G$-action has rank 3.
Then $G_{B,B'}= (\la\tau\ra\rtimes\la \sigma\ra)\cap G$ has a unique transitive representation of degree $m$, contradicting (D). Hence the action does not have rank 3.

Assume now that we are in the second case, that is $|B|=2$, and that the $G$-action has rank 3.
Let $H=G\cap\PGL(2,q)$. Then either $q$ is odd or $|G/H|$ is even, or both.
By property (D),  $G_{B,B'}$ has a normal subgroup $K$ such that  $G_{B,B'}/K\cong C_2\times C_2$. In particular, the Sylow $2$-subgroup of $G_{B,B'}$ is not cyclic and hence
 both $q-1$ and $|G/H|$ are even. Hence $e$ is even, and so $q\equiv 1 \pmod 4$. Note that  $|G/H|=e/i$ for some divisor $i$ of $e$ such that $e/i$ is even.
 More precisely, $G_B$, $G_{B,B'}$, $G_{(B)}$ and $G_{(B),B'}$ are as in one of the lines of Table \ref{PSL-2}. In lines (1.b), (1.c), (2.b), (2.c), there are additional conditions on $p$ and $i$  for the subgroup $G_{(B)}$ to have index 2 in $G_B$ (the conditions we need will be described below).

\begin{table}
\begin{center}
\begin{tabular}{|l|l|l|l|l|}%\label{PSU-2}
\hline 
$G_B$& $G_{B,B'}$&$G_{(B)}$&$G_{(B),B'}$&\\
\hline
$ Q\rtimes\la\tau^2,\mu\ra$& $\la\tau^2,\mu\ra$&$ Q\rtimes\la\tau^2,\mu^{2}\ra$& $\la\tau^2,\mu^{2}\ra$&(1.a)\\
where $\mu=\sigma^i$, or $\sigma^i\tau$& &$ Q\rtimes\la\tau^4,\mu\ra$& $\la\tau^4,\mu\ra$&(1.b)\\
                                       &&$ Q\rtimes\la\tau^4,\mu\tau^2\ra$& $\la\tau^4,\mu\tau^2\ra$&(1.c)\\
\hline
$(Q\rtimes\la\tau\ra)\rtimes\la\sigma^i\ra$&$\la\tau\ra \rtimes \la\sigma^i\ra$ &$(Q\rtimes\la\tau\ra)\rtimes\la\sigma^{2i}\ra$&  $\la\tau\ra\rtimes\la\sigma^{2i}\ra$&(2.a)\\
&& $(Q\rtimes\la\tau^2\ra)\rtimes\la\sigma^i\ra$& $\la\tau^2\ra\rtimes\la\sigma^{i}\ra$&(2.b)\\
&& $Q\rtimes\la\tau^2,\sigma^i\tau\ra$& $\la\tau^2,\sigma^{i}\tau\ra$&(2.c)\\
\hline
\end{tabular}\\
\caption{Possibilities for the proof of Proposition \ref{PSL2}} \label{PSL-2}
\end{center}
\end{table}

If $G=\langle H,\mu\rangle$, where $\mu=\sigma^i$ or $\sigma^i\tau$ for some $i$ dividing $e$, then $\langle H,\mu^2 \rangle$ is normal in $G$, and so, by property (F), $G_{(B)}$ cannot be equal to $\langle H,\mu^2
\rangle\cap
G_B$. This eliminates cases (1.a) and (2.a).

In all the other cases, $G_{(B),B'}=\langle\tau^{2\ell},\sigma^{i}\tau^k\rangle$ for some $k\in [0,2\ell-1]$ where $\ell=2$ in cases (1.b) and (1.c), and $\ell=1$ in cases (2.b) and (2.c).
Recall that $g$ interchanges $B$ and $B'$, $\tau^g=\tau^{-1}$ and $\sigma^g=\sigma$.  Thus  $G_{(B'),B}= (G_{(B),B'})^g$ is generated by $(\tau^{2\ell})^g=\tau^{-2\ell}$ and $(\sigma^{i}\tau^k)^g=\sigma^{i}\tau^{-k}$.
Obviously, we have  $\langle\tau^{2\ell}\rangle=\langle(\tau^{2\ell})^g\rangle$. 
If either $k=0$ or $k=\ell\in \{1,2\}$, then $\sigma^{i}\tau^{-k}=(\sigma^i\tau^{k})\tau^{-2k}\equiv \sigma^i\tau^{k}$ modulo $ \langle\tau^{2\ell}\rangle$. It follows that  $G_{(B'),B}= G_{(B),B'}$ and hence $G_{(B),B'}$ fixes $B\cup B'$ pointwise and property (A) fails.

Therefore we must have  $\ell=2$ and $k=1$ or $3$, that is, $G=\la \PSL(2,q),\sigma^i\tau\rangle$ and we are in case (1.b) or (1.c) respectively.
We have  $(\sigma^i\tau^j)^{(e/i)}=\tau^{jc}$, where $c=1+p^{i}+p^{2i}+\ldots +p^{(e/i-1)i}$ is even since $e/i$ is even, and hence $\tau^{c}$ is contained in $\la \tau^2\ra\leqslant \PSL(2,q)$. Thus $G$ does not contain $\PGL(2,q)$.
We must have that $G_{(B)}= Q\rtimes\la\tau^4,\sigma^i\tau^k\ra$ has index 2 in $G_B=Q\rtimes\la\tau^2,\sigma^i\tau\ra$, and so  $(\sigma^i\tau^k)^{(e/i)}=\tau^{kc}$ must be contained in  $\langle\tau^{4}\rangle$, that is, $4$ must divide $c$. Since $e/i$ is even, this happens exactly when $p^i\equiv 3\pmod 4$ or $p^i\equiv 1\pmod 4$ and $e/i\equiv 0 \pmod 4$. Hence all the stated conditions are satisfied.
%Conversely, when one of these conditions is satisfied, then $(\tau^k\sigma^i)^t\in\langle\tau^{4}\rangle$ for all $t$ multiple of $e/i$ and so $\tau^2\notin\langle\tau^{4},\tau^k\sigma^i\rangle$.
Now assume $G$ satisfies these conditions. 
Then $G_{(B),B'}$ has index 2 in $G_{B,B'}$, so $G_{B,B'}$ is transitive on $B$.
Also $G_{(B'),B}= (G_{(B),B'})^g=\la\tau^4,\sigma^i\tau^{4-k}\ra\neq G_{(B),B'}$, which means that $G_{(B),B'}$ does not fix $B'$ pointwise and hence is transitive on $B'$. Therefore property (A) is satisfied.
 Since $Q\leq G_{(B)}$, we have that $G_{(B)}$ is transitive on the blocks distinct from
$B$, therefore property (B) is satisfied. By Lemma \ref{properties}, these examples have indeed a rank 3 action.

Moreover, these two actions (with $k=1$ and $k=3$) are not isomorphic. If they were isomorphic, then we would have $Q\rtimes\la\tau^4,\sigma^i\tau\ra$ and $Q\rtimes\la\tau^4,\sigma^i\tau^3\ra$ conjugate in $G$.
Since property (B) is satisfied in both cases, the only block-stabiliser they are contained in is $G_B$, and so these subgroups would be conjugate in $G_B$. These two subgroups have index 2 (hence are normal) in $G_B$, and so cannot be conjugate in $G_B$.
\qed\medskip

We next consider the groups with socle $\PSL(a,q)$ for $a>2$, and describe in the following remark the model we are going to use. It is followed by a technical lemma.
\begin{remark}\label{PSLadescr}{\rm
Let $T=\PSL(a,q)\leq G\leq \PGammaL(a,q)$, with $G$ acting on $\calB$ of size $\frac{q^a-1}{q-1}$, where $q=p^e$ with $p$ a prime and $e\geq 1$. 
Then $G$ is a subgroup of $\PGL(a,q) \rtimes \la \sigma^i\ra$, where $\sigma$ is the Frobenius automorphism and $i$ divides $e$. 
If $\bfZ$ is the set of scalar matrices in $\GL(a,q)$, then the elements of $\PGL(a,q)$ are the cosets $A \bfZ$, where $A$ is a matrix in $\GL(a,q)$. When convenient we will single out a representative for a coset.
Let $K=\{x^a|x \in \GF(q)^*\}$, a subgroup of $(\GF(q)^*,.)$ of order $(q-1)/(a,q-1)$. Then $\PSL(a,q)=\{A \bfZ |\Det(A)\in K\}$.
We have $H:=G\cap \PGL(a,q)=\{A\bfZ |\Det(A)\in F\}$, where $F$ is a subgroup of $(\GF(q)^*,.)$ containing $K$.
Let $\omega$ be a primitive element of $\GF(q)$. Then $\GF(q)^*=\la \omega\ra$ and  $K=\la \omega^{(a,q-1)}\ra$.
Let $d:=\min\{j>0|\omega^j\in F\}$. Then  $d$ divides $(a,q-1)$ and $F=\la \omega^d\ra$, as $(\GF(q)^*,.)$ is cyclic. Notice that $H$ has index $d$ in $\PGL(a,q)$.

Let $\tau$ be the diagonal $(a\times a)$-matrix with $\tau_{jj}=1$ for $j<a$ and $\tau_{aa}=\omega$. Then either $G=H$ or $G=\la H,  \sigma^i\tau^r\bfZ\ra$ for some $r\in \{0,\ldots,d-1\}$  since $\tau^d\bfZ\in H$ and $\PGL(a,q)=\la H,\tau\bfZ\ra$. We recall that $i$ dvides $e$.
For $k\geq 1$, $( \sigma^i\tau^r\bfZ)^k=\sigma^{ki} \tau^{r(1+p^{i}+\ldots +p^{(k-1)i})}\bfZ$, and so $(\sigma^i \tau^r\bfZ )^{e/i}=\tau^{r(1+p^{i}+\ldots p^{(e-i)})}\bfZ=\tau^{r\frac{q-1}{p^i-1}}\bfZ$. This element must be in $H$, and so its determinant must be in $\la\omega^d\ra$. Therefore we have that $d$ divides $r\frac{q-1}{p^i-1}$.

We identify $\calB$ with the set of $1$-dimensional subspaces of the vector space $\GF(q)^a$ of row vectors with standard basis $e_1,e_2,\ldots, e_a$.
Let $B=\la e_1\ra$.
Then $H_B$ consists of the elements $A\bfZ $ of $H$ such that $A_{1j}=0$ for all $j\geq 2$. Each coset $A\bfZ$ contains a unique representative $A$  with $A_{11}=1$.
More precisely, let $g_{\underline{b},X}=
\begin{pmatrix}
1 &   \underline{0}^T \\
\underline{b}  & X \\
 \end{pmatrix}
$, where $\underline{b},\underline{0}\in \GF(q)^{(a-1)\times 1}$ and $X\in \GL(a-1,q)$. Then $H_B=\{ g_{\underline{b},X}\bfZ|\Det(X)\in F\}$, and $G_B=\la H_B,\sigma^i \tau^r\bfZ\ra$. Thus $G_B$ is isomorphic to a subgroup between $\ASL(a-1,q)$ and $\AGammaL(a-1,q)$. }
\end{remark}

\begin{lemma}\label{PSL-technical}
Let $Z=\la y\ra\cong C_t$, $t>1$, and let $n$ be a positive integer. Let $Z_0=\la y^s\ra$, where $s$ is a divisor of $t$.
Then the number of $x\in Z$ such that $x^n\in Z_0$ is $t(n,s)/s$. 
\end{lemma}
\proof 
Let $Z_1=\la y^n\ra=\la y^{(n,t)}\ra$ of order $t/(n,t)$. Then $Z_1\cap Z_0=\la y^u\ra$ where $u=\lcm\{s,(n,t)\}=\frac{s(n,t)}{(s,(n,t))}=\frac{s(n,t)}{(n,s)}$ since $s|t$.
The map $\varphi:y^j\rightarrow y^{jn}$ is a homomorphism from $Z$ to $Z$ with image $Z_1$ and kernel $\la y^{\frac{t}{(n,t)}}\ra$ of order $(n,t)$. Thus each element of $Z_1$ is the image of exactly $(n,t)$ elements of $Z$ under $\varphi$. Hence the number  we seek is the number of elements $y^j\in Z$ such that $\varphi(y^j)\in Z_0\cap Z_1$, and so is equal to $(n,t)| Z_0\cap Z_1|=(n,t)\frac{t}{u}=(n,t)t.\frac{(n,s)}{s(n,t)}=\frac{t(n,s)}{s}$.
\qed\medskip

\begin{prop}
Let $G$ be an imprimitive permutation group on $\Omega$ with system of imprimitivity $\calB$ such that  $G\cong G^\calB$ is as in (h) of Theorem $\ref{2transgroupsAS}$ for $a\geq 3$.
Then $G$ satisfies Hypothesis $\ref{mainhypo}$ if and only if
 $G$, $n$ and $|B|$ are as in one of Lines $10$ to $20$ of Table $\ref{bigtable2}$, with column $5$ giving additional conditions on $G$ or $|B|$. 
Moreover, $G$ has rank $3$ on $\Omega$ if and only if one of the following  holds :
\begin{enumerate}
\item $G$ satisfies the conditions on Line $12$ of  Table $\ref{bigtable2}$ and $(md,a)=d$.
 \item  $G=\PGL(3,4)$ or $\PGammaL(3,4)$, $|B|=6$, $G_B^B\cong \PSL(2,5)$ or $\PGL(2,5)$ respectively.
 \item  $G=\PSL(3,5)$, $|B|=5$, $G_B^B\cong S_5$.
\item  $G=\PSL(5,2)$, $|B|=8$,  $G_B^B\cong A_8$.
\item  $G=\PGammaL(3,8)$, $|B|=28$,  $G_B^B\cong \Ree(3)$.
\item $G=\PSL(3,2)$, $|B|=2$,  $G_B^B\cong C_2$.
\item $G=\PSL(3,3)$, $|B|=3$ , $G_B^B\cong S_3$.
\end{enumerate}
\end{prop}

\proof
Let $T=\PSL(a,q)\leq G\leq \PGammaL(a,q)=A$. We use the notation of Remark \ref{PSLadescr}.  
Let  $B$ and $B'$, identified respectively with $\la e_1\ra$ and $\la e_2\ra$, be two blocks of $\calB$.
Let $$h_{x,\underline{c},\underline{d},Y}=
\begin{pmatrix}
1 &  0& \underline{0}^T \\
0&x&\underline{0}^T\\
\underline{c}&\underline{d}  &Y \\
 \end{pmatrix},$$
 where $x\in \GF(q)^*$, $\underline{c},\underline{d},\underline{0}\in \GF(q)^{(a-2)\times 1}$ and $Y\in \GL(a-2,q)$.
Then  $H_{B,B'}$ consists of the elements  $h_{x,\underline{c},\underline{d},Y}\bfZ$ that lie in $H$, that is, those for which $x\,\Det(Y)\in F$. The subgroup $\la \sigma^i \tau^r\bfZ\ra$ is in $G_{B,B'}$, so  $G_{B,B'}= \la H_{B,B'},\sigma^i \tau^r\bfZ\ra$.  Let  $g\in \SL(a,q)$ such that $g_{12}=1, g_{21}=-1, g_{jj}=1$ for $j\geq 3$ and all other $g_{jk}=0$.
Then $g\bfZ$ is an element of $H$ interchanging $B$ and $B'$.

The subgroup $M=\{ g_{\underline{b},I}\bfZ|\underline{b}\in\GF(q)^{(a-1)\times 1}\}$ is the unique minimal normal subgroup of the stabiliser $G_B$ .
Suppose that $G_B$ acts faithfully on $B$. The only 2-transitive action is affine, with $|B|=|M|=q^{a-1}$, as in Line 10 of Table \ref{bigtable2}, and for some $v\in B$, $G_v=\la S,\tau^d\bfZ,\sigma^i \tau^r\bfZ\ra$ where $S=\{g_{\underline{0},X}\bfZ|\Det(X)=1\}$. Notice that $\la S,\tau^d\bfZ\ra=\{g_{\underline{0},X}\bfZ|\Det(X)\in F\}$. From the description of $H_{B,B'}$, it follows that the index of $G_{B,B',v}$ in $G_{B,B'}$ is $q^{a-2}$, and so  $G_{B,B'}$ is not transitive on $B$, contradicting property (A). So this $G$-action does not have rank 3.

Suppose now that $G_B$ does not act faithfully on $B$, and so $M\leqslant G_{(B)}$. Notice that $G_B/M$ is isomorphic to a subgroup of $\GammaL(a-1,q)$.
Let $N=\la S\ra\cong \SL(a-1,q)$, where $S$ is as above.
 Then $MN\lhd G_B$. 

\underline{Suppose first that  $MN\leqslant G_{(B)}$.} \\Then $G_B^B$ is isomorphic to a quotient of $\la \tau^d\bfZ, \sigma^i \tau^r\bfZ\ra\leqslant
\la \tau\bfZ, \sigma^i\ra\cong C_{q-1}\rtimes C_{e/i}$. Hence either $|B|=2$, or $|B|$ is a prime dividing $(q-1)/d$.

Case (1): Here $|B|=2$ and $G_{B,v}=G_{(B)}=\la H_B, (\sigma^i \tau^r\bfZ)^2\ra$, where $G/H$ has even order, as in Line $11$ of Table \ref{bigtable2}. Then $G_v=G_B\cap \la H,(\sigma^i \tau^r\bfZ)^2\ra$ and $\la H,(\sigma^i \tau^r\bfZ)^2\ra$ is a normal subgroup of $G$. Thus property (F) fails, so this action does not have rank $3$.

Case (2): Here $|B|=m$ is a prime dividing $(q-1)/d$ and $G_{B,v}=\la M,N ,\tau^{dm}\bfZ,\sigma^i \tau^{r+\lambda d}\bfZ \ra$ for some $\lambda\in [0,m-1]$.  
In other words $H_{B,v}$ consists of the elements $g_{\underline{b},X}\bfZ$ such that $\Det(X)\in \la \omega^{dm}\ra$. 
In order for $G_{B,v}$ to be of index $m$ in $G_B$, it is necessary that  $(\sigma^i \tau^{r+\lambda d}\bfZ )^{e/i}\in \la \tau^{dm}\bfZ \ra$, that is $dm\mid (r+\lambda d)\frac{q-1}{p^i-1}$. Moreover, in order that $G^B_B$ be 2-transitive, it is also needed that the order of $p^i$ modulo $m$ is $m-1$ (for $m=2$ that just means that $p$ is odd). Thus Line 12 of Table \ref{bigtable2} holds.

Since $G_v$ contains $MN$, $G_v$ induces a transitive action on ${\mathcal B}\setminus\{B\}$, so property (B) is satisfied.
Since $\sigma^i \tau^{r+\lambda d}\bfZ$ fixes $B$, $B'$, $v$ and also $v^g$ (note that $\tau^g=\tau$ and $\sigma^g=\sigma$), we have that  $|G_{B,B'}:G_{B,B',v}|=|H_{B,B'}:H_{B,B',v}|$ and $|G_{B,B',v}:G_{B,B',v,v^g}|=|H_{B,B',v}:H_{B,B',v,v^g}|$. We have $H_{B,B',v}=\{h_{x,\underline{c},\underline{d},Y}\bfZ|x\,\Det(Y)\in \la \omega^{md}\ra\}$, which has index $m$ in $H_{B,B'}$, and so $G_{B,B'}$ is transitive on $B$. Therefore property (A) will be satisfied if and only if $G_{B,B',v}$ is transitive on $B'$, that is if and only if $G_{B,B',v,v^g}$ has index $m$ in $G_{B,B',v}$. We  now compute this index, which we have seen is equal to $|H_{B,B',v}:H_{B,B',v,v^g}|$. 
We have $H_{B,B',v}=\{h_{x,\underline{c},\underline{d},Y}\bfZ|x\,\Det(Y)\in \la \omega^{md}\ra\}$. Also $H_{B,B',v^g}=H_{B,B',v}^g=\{h_{1/x,(-1/x)\underline{d},(1/x)\underline{c},(1/x) Y}\bfZ|x\,\Det(Y)\in \la \omega^{md}\ra\}=\{h_{x',\underline{c'},\underline{d'},Y'}\bfZ|x'^{(1-a)}\,\Det(Y')\in \la \omega^{md}\ra\}$. Therefore $H_{B,B',v,v^g}=\{h_{x,\underline{c},\underline{d},Y}\bfZ|x\,\Det(Y),x^a\in \la \omega^{md}\ra\}$.
For a chosen $x$, the number of matrices $Y$ such that $x\,\Det(Y)\in \la \omega^{md}\ra$ does not depend on the choice of $x$. Say that number is $k$. The number of choices for $\underline{c},\underline{d}$ also does not depend on $x$. Hence $|H_{B,B',v}|=(q-1)q^{2(a-2)}k$. For $H_{B,B',v,v^g}$ the only allowable choices for $x$ are the ones whose $a^\text{th}$ power is in $\la \omega^{md}\ra$. By Lemma \ref{PSL-technical}, the number of $x\in\la\omega\ra$ such that $x^a\in\la\omega^{md}\ra$ is $J=\frac{(q-1)(md,a)}{md}$. Thus  $|H_{B,B',v,v^g}|=Jq^{2(a-2)}k$, and  $|H_{B,B',v}:H_{B,B',v,v^g}|=(q-1)/J=\frac{md}{(md,a)}$. Therefore property (A) is satisfied if and only if $(md,a)=d$. Thus, by Lemma \ref{properties}, this $G$-action has rank $3$ if and only if $(md,a)=d$, as in (1) of the statement.

\underline{We now assume that  $N\not\leqslant G_{(B)}$.}\\
Let $C=\{g_{\underline{0},\nu I}\bfZ|\nu^{a-1}\in F\}$, where $F=\la\omega^d\ra$ as in Remark \ref{PSLadescr}. By  Lemma \ref{PSL-technical}, the number of $\nu\in\la \omega\ra$  such that $\nu^{a-1}\in \la \omega^{d}\ra$ is  $\frac{(q-1)(a-1,d)}{d}$. So $C\cong C_{\frac{q-1}{d}(a-1,d)}$.
We claim that $C\leqslant G_{(B)}$.

Suppose to the contrary that $C^B\neq 1$. Then $C^B$ is a cyclic normal subgroup of the $2$-transitive group $G_B^B$, and hence $|C^B|$ is a prime $p'$ and $G_B^B\leqslant \AGL(1,p')$. Moreover $C^B$ is self-centralising in $G_B^B$ by Lemma \ref{GBBaffine}. However $N\cong \SL(a-1,q)$ centralizes $C$ and by assumption $N^B\neq 1$, so $N^B\leqslant C^B$ is cyclic. Thus $\SL(a-1,q)$ has a nontrivial cyclic quotient and hence $a=3$ and $q\leq 3$. Since $p'$ divides $q-1$, it follows that $q=3$ and $p'=2$. However the only cyclic quotient of $N\cong \SL(2,3)=Q_8.C_3$ has order $3$, not $2$. This contradiction proves the claim.

Hence $C\leqslant G_{(B)}$, and so $MC \lhd G_{(B)}$. We have $\PSL(a-1,q)\leqslant G_B/(MC) \leqslant\PGammaL(a-1,q)$, and  $G_B^B=G_B/G_{(B)}$ is isomorphic to $(G_B/(MC))/(G_{(B)}/(MC))$. 
If $(a-1,q)\neq (2,2)$ or $(2,3)$, then every nontrivial normal subgroup of $G_B/(MC)$ must contain $\PSL(a-1,q)$. However, $(G_{(B)}/(MC))$ cannot contain $\PSL(a-1,q)$, as $G_{(B)}$ does not contain $N$. Therefore  $G_{(B)}=MC$ or  $(a-1,q)=(2,2)$ or $(2,3)$.

Assume first that $G_{(B)}=MC$. Then $\PSL(a-1,q)\leqslant G_B^B\leqslant\PGammaL(a-1,q)$.
Hence $G_B^B$ has one natural 2-transitive action of degree $(q^{a-1}-1)/(q-1)$ if $a=3$ and two such actions otherwise (the actions on points and hyperplanes of the projective space), as on Line 13 of Table \ref{bigtable2}.
For the action on points, we have $H_{B,v}=\{ g_{\underline{b},X}\bfZ|\Det(X)\in F, X_{1,i}=0 \text{ for all }i>1\}$ for some $v\in B$, and so $G_{B,v}=\la H_{B,v}, \sigma^i\tau^r\bfZ \ra$  contains $G_{B,B'}$. Thus property (A) fails and this $G$-action does not have rank 3. 
If $a\geq 4$, there is also the action on hyperplanes, for which  $H_{B,v}=\{ g_{\underline{b},X}|\Det(X)\in F, X_{i,1}=0 \text{ for all }i>1\}$ for some $v\in B$. Hence  $H_{B,B',v}=\{h_{x,\underline{c},\underline{0},Y}\bfZ|x\,\Det(Y)\in F\}$. Since  $\sigma^i\tau^r\bfZ$ is in $G_{B,B',v}$ as well as in $G_{B,B'}$, we have $|G_{B,B'}:G_{B,B',v}|=|H_{B,B'}:H_{B,B',v}|=q^{a-2}\neq (q^{a-1}-1)/(q-1)$, and so $G_{B,B'}$ does not act transitively on $B$, and property (A) fails.  Thus this $G$-action does not have rank $3$ either.

Now we look at the other possible 2-transitive actions of $G_B^B$ having simple socle $\PSL(a-1,q)$. 
By Theorem \ref{2transgroupsAS}, these occur for $(a,q,|B|)=(3,4,6)$, $(3,5,5)$,  $(5,2,8)$
, $(3,8,28)$, $(4,2,8)$, $(3,7,7)$, $(3,9,6)$, and $(3,11,11)$. We consider each of these below.  For simplifying the proof we sometimes cie an easily repeatable computation with Magma \cite{magma}.
\begin{enumerate}
\item $T=\PSL(3,4)$. Here $H=T$ or $H=\PGL(3,4)$. In both cases $H_B/(MC)\cong \PSL(2,4)\cong \PSL(2,5)$, which has a $2$-transitive action on $6$ points, as on Line 14 of Table \ref{bigtable2}. If $H=T$, then $|G_{B,B'}|=48$ or $96$ which is not divisible by $36$, so property (C) fails. If $H=\PGL(3,4)$, that is, if $G=\PGL(3,4)$ or $\PGammaL(3,4)$, then Magma confirms we have a rank $3$ example, as in (2) of the statement, and $G_B^B\cong \PSL(2,5)$ or $\PGL(2,5)$ respectively.\\
\item $T=\PSL(3,5)$. Since there is no field automorphism and $(a,q-1)=1$, we have $G=\PSL(3,5)$. Then $G_B/(MC)\cong \PGL(2,5)\cong S_5$, which has a $2$-transitive action on $5$ points, as on Line 15 of Table \ref{bigtable2}. Magma confirms this gives a rank $3$ example, as in (3) of the statement. \\
\item $T=\PSL(5,2)$. Since there is no  field automorphism and $(a,q-1)=1$, we have $G=\PSL(5,2)$. Then $G_B/(MC)\cong \PGL(4,2)\cong A_8$, which has a $2$-transitive action on $8$ points, as on Line 16 of Table \ref{bigtable2}. Magma confirms this gives a rank $3$ example, as in (4) of the statement. \\
\item  $T=\PSL(3,8)$. Since $(a,q-1)=1$, we have $d=1$,  $G_B/(MC)\geqslant H_B/(MC)\cong \PGL(2,8)$. If $G=\PGammaL(3,8)$, then $G_B/(MC)\cong \PGammaL(2,8)\cong \Ree(3)$, which has a $2$-transitive action on $28$ points, as on Line 17 of Table \ref{bigtable2}. Magma confirms this gives a rank $3$ example, as in (5) of the statement. \\
\item $T=\PSL(4,2)$. Since there is no field automorphism and $(a,q-1)=1$, we have $G=\PSL(4,2)$. Then $G_B/(MC)\cong \PGL(3,2)\cong \PSL(2,7)$, which has a $2$-transitive action on $8$ points, as in Line 18 of Table \ref{bigtable2}. However, $|G_{B,B'}|=96$ which is not divisible by $64$, hence property (C) fails. Thus this $G$-action does not have rank $3$.\\
\item $T=\PSL(3,7)$, $G=T$ or $G=\PGL(3,7)$. In both cases $H_B/(MC)\cong \PGL(2,7)$ which does not have a $2$-transitive action on $7$ points, since the isomorphism between $\PSL(2,7)$ and $\PSL(3,2)$ does not extend to $\PGL(2,7)$.\\
\item $T=\PSL(3,9)$, $H=T$. Then $G_B/(MC)\geqslant H_B/(MC)\cong \PGL(2,9)$ which does not have a $2$-transitive action on $6$ points, since the isomorphism between $\PSL(2,9)$ and $A_6$ does not extend to $\PGL(2,9)$.\\
\item $T=\PSL(3,11)$, $G=T$. Then $G_B/(MC)\cong \PGL(2,11)$ which does not have a $2$-transitive action on $11$ points, since the action of $\PSL(2,11)$ on $11$ points does not extend to $\PGL(2,11)$.\\
\end{enumerate} 

Finally assume that $(a-1,q)=(2,2)$ or $(2,3)$ and $MC < G_{(B)}$. We recall that we are also assuming $N\not\leqslant G_{(B)}$.
Consider first $G=\PSL(3,2)$. Then $C$ is trivial and $G_B=M\rtimes N\cong 2^2\rtimes S_3$. Hence $G_{(B)}= M \rtimes D$ where $D\leqslant N$ is isomorphic to $C_3$. Then $|B|=2$, as on Line 19 of Table \ref{bigtable2}. 
Magma confirms this gives a rank $3$ example, as in (6) of the statement.
Consider now $G=\PSL(3,3)$.  Then $C\cong C_2$. As $MN$ is the unique maximal subgroup of $G_B$, $MC< G_{(B)}< MN\cong 3^2\rtimes \SL(2,3)$. Hence $G_{(B)}=M\rtimes D$ where $D\leqslant N$ is isomorphic to $Q_8$. Then $G_B^B\cong S_3$ (since $N^B\neq 1$) and this has a $2$-transitive action on $3$ points, as on Line 20 of Table \ref{bigtable2}. 
We have that $G_{(B)}$ is transitive on the 12 other blocks (property (B)). Moreover $|G_{(B),B'}|=6$  and $|G_{(B),(B')}|=1$, and so $G_{(B),B'}^{B'}\cong S_3$, and (A) is satisfied. Hence this $G$-action has rank 3, as in (7) of the statement. This completes the proof.\qed


\medskip









\begin{thebibliography}{1}




\bibitem{gridpaper}
John Bamberg, Geoffrey Pearce, and Cheryl~E. Praeger.
\newblock Transitive decompositions of graph products: rank 3 grid type.
\newblock {\em J. Group Theory}, 11(2), 185--228 (2008).

\bibitem{Bannai}
E. Bannai.
\newblock {\em Maximal subgroups of low rank of finite symmetric and alternating groups}.
\newblock J. Fac. Sci. univ. Tokyo, Sect. IA 18, 475--486 (1972). 

\bibitem{BMMN}
Meenaxi Bhattacharjee, Dugald Macpherson, Rögnvaldur G.~M\"oller, and Peter M.~Neumann.
\newblock Notes on infinite permutation groups. 
\newblock Texts and Readings in Mathematics, 12. Lecture Notes in Mathematics, 1698. Hindustan Book Agency, New Delhi; co-published by Springer-Verlag, Berlin, 1997. 


\bibitem{magma}
W.~Bosma, J.~Cannon and C.~Playoust, The Magma algebra system
  I: The user language, J. Symb. Comp. 24 3/4 (1997)
235-265. Also see the {\sc Magma} home page at
http://www.maths.usyd.edu.au:8000/u/magma/. 



\bibitem{Burnside}
W.~Burnside.
\newblock {\em Theory of groups of finite order}.
\newblock Dover Publications Inc., New York, 1955.
\newblock 2d ed.

\bibitem{CameronBook}
Peter~J. Cameron.
\newblock {\em Permutation groups}, volume~45 of {\em London Mathematical
  Society Student Texts}.
\newblock Cambridge University Press, Cambridge, 1999.


\bibitem{atlas}
J.~H. Conway, R.~T. Curtis, S.~P. Norton, R.~A. Parker, and R.~A. Wilson.
\newblock {\em Atlas of finite groups}.

\newblock Oxford University Press, Eynsham, 1985.
\newblock Maximal subgroups and ordinary characters for simple groups, With
  computational assistance from J. G. Thackray.

\bibitem{Dempwolff01}
U.~Dempwolff.
\newblock Primitive rank 3 groups on symmetric designs.
\newblock {\em Des. Codes Cryptogr.}, 22(2), 191--207 (2001).

\bibitem{Dempwolff04}
U.~Dempwolff.
\newblock Affine rank 3 groups on symmetric designs.
\newblock {\em Des. Codes Cryptogr.}, 31(2), 159--168 (2004).

\bibitem{Dempwolff06}
U.~Dempwolff.
\newblock Symmetric rank 3 designs with regular, elementary abelian, normal
  subgroups.
\newblock In {\em Finite geometries, groups, and computation}, 67--73.
  Walter de Gruyter GmbH \& Co. KG, Berlin, 2006.

\bibitem{Devillers}
Alice Devillers.
\newblock A classification of finite partial linear spaces with a rank 3
  automorphism group of grid type.
\newblock {\em European J. Combin.}, 29(1), 268--272 (2008).

\bibitem{Devillers05}
Alice Devillers.
\newblock A classification of finite partial linear spaces with a primitive
  rank 3 automorphism group of almost simple type.
\newblock {\em Innov. Incidence Geom.}, 2, 129--175 (2005).

\bibitem{Devillers06}
Alice Devillers and J.~I. Hall.
\newblock Rank 3 {L}atin square designs.
\newblock {\em J. Combin. Theory Ser. A}, 113(5), 894--902 (2006).

\bibitem{LDT}
A. Devillers, M. Giudici, C. H. Li and C. E. Praeger, 
Locally $s$-distance transitive graphs, submitted.

\bibitem{DM}
John~D. Dixon and Brian Mortimer.
\newblock {\em Permutation groups}, volume 163 of {\em Graduate Texts in
  Mathematics}.
\newblock Springer-Verlag, New York, 1996.


\bibitem{GodPra}
Chris~D. Godsil, Robert~A. Liebler, and Cheryl~E. Praeger,
\newblock Antipodal distance transitive covers of complete graphs.
\newblock {\em European J. Combin.}, 19(4), 455--478 (1998).

\bibitem{Higman}
Donald G.~Higman, \newblock Finite permutation groups of rank $3$. 
\newblock {\em Math. Z.} 86, 145--156 (1964).

%\bibitem{Kantor}
%William~M. Kantor.
%\newblock Homogeneous designs and geometric lattices.
%\newblock {\em J. Combin. Theory Ser. A}, 38(1),  66--74 (1985).

\bibitem{KanLieb}
William~M. Kantor, Robert A.~Liebler
\newblock The rank $3$ permutation representations of the finite classical groups. 
\newblock Trans. Amer. Math. Soc.  271(1),  1--71 (1982).

\bibitem{KLM}
Gregor Kemper, Frank L\"{u}beck and Kay Magaard,
Matrix generators for the Ree groups ${}\sp 2G\sb 2(q)$.
Comm. Algebra 29,  407--413 (2001). 

\bibitem{Liebeck}
Martin~W. Liebeck.
\newblock The affine permutation groups of rank three.
\newblock {\em Proc. London Math. Soc. (3)}, 54(3), 477--516 (1987).

\bibitem{LiebeckSaxl}
Martin~W. Liebeck and Jan Saxl.
\newblock The finite primitive permutation groups of rank three.
\newblock {\em Bull. London Math. Soc.}, 18(2), 165--172 (1986).

\bibitem{Mortimer}
Brian Mortimer, The modular permutation representations of the known doubly
transitive groups,
Proc. London Math. Soc. (3) 41, 1--20 (1980).

\bibitem{ONan}
Michael~E. O'Nan,
\newblock Automorphisms of unitary block designs.
\newblock {\em J. Algebra}, 20, 495--511 (1972).

\bibitem{Geoff}
 Cheryl~E. Praeger and Geoffrey Pearce,
\newblock Rank 3 transitive decompositions of complete multipartite graphs,
\newblock submitted.

\bibitem{Ree}
Rimhak Ree, A family of simple groups associated with the simple Lie algebra of type $(G\sb{2})$.  Amer. J. Math.  83, 432--462 (1961).

%\bibitem{Schmidt}
%Roland Schmidt,
%\newblock {\em Subgroup lattices of groups}, volume~14 of {\em de Gruyter
%  Expositions in Mathematics}.
%\newblock Walter de Gruyter \& Co., Berlin, 1994.

\bibitem{suz}
Michio Suzuki,  On a class of doubly transitive groups,  Ann. of Math. (2)  75,  105--145 (1962). 

\bibitem{ward}
Harold N. Ward,  On Ree's series of simple groups, Trans. Amer. Math. Soc.  121,  62--89 (1966).


\end{thebibliography}
\end{document}